\documentclass[oneside,english]{amsart}
\usepackage{lmodern}
\usepackage[T1]{fontenc}
\usepackage{color}
\usepackage{babel}
\usepackage{varioref}
\usepackage{amsthm}
\usepackage{amstext}
\usepackage{amssymb}
\usepackage{esint}
\usepackage[unicode=true,
 bookmarks=true,bookmarksnumbered=false,bookmarksopen=false,
 breaklinks=false,pdfborder={0 0 1},backref=false,colorlinks=false]
 {hyperref}
\hypersetup{pdftitle={Enforcing non-zero constraints in PDEs and applications},
 pdfauthor={Giovanni S. Alberti}}

\makeatletter
\theoremstyle{definition}
  \newtheorem{defn}{\protect\definitionname}
  \theoremstyle{definition}
  \newtheorem{example}{\protect\examplename}
  \theoremstyle{plain}
  \newtheorem{thm}{\protect\theoremname}
  \theoremstyle{remark}
  \newtheorem{rem}{\protect\remarkname}
  \theoremstyle{plain}
  \newtheorem{cor}{\protect\corollaryname}
  \theoremstyle{plain}
  \newtheorem{lem}{\protect\lemmaname}
 \ifx\proof\undefined\
   \newenvironment{proof}[1][\proofname]{\par
     \normalfont\topsep6\p@\@plus6\p@\relax
     \trivlist
     \itemindent\parindent
     \item[\hskip\labelsep
           \scshape
       #1]\ignorespaces
   }{%
     \endtrivlist\@endpefalse
   }
   \providecommand{\proofname}{Proof}
 \fi
  \theoremstyle{plain}
  \newtheorem{prop}{\protect\propositionname}

\usepackage{empheq}

\usepackage{tikz}
\usepackage{standalone}

\usepackage{enumitem}
\setlist{leftmargin=*}

\@ifundefined{showcaptionsetup}{}{%
 \PassOptionsToPackage{caption=false}{subfig}}
\usepackage{subfig}
\makeatother

  \providecommand{\corollaryname}{Corollary}
  \providecommand{\definitionname}{Definition}
  \providecommand{\examplename}{Example}
  \providecommand{\lemmaname}{Lemma}
  \providecommand{\propositionname}{Proposition}
  \providecommand{\remarkname}{Remark}
\providecommand{\theoremname}{Theorem}

\begin{document}
\global\long\def\R{\mathbb{R}}

\global\long\def\N{\mathbb{N}}

\global\long\def\C{\mathbb{C}}

\global\long\def\Cl{C}

\global\long\def\tr{\mathrm{tr}}

\global\long\def\Se{\mathcal{S}}

\global\long\def\P{L_{\beta_{1},\beta_{2}}^{\infty}(\Omega)}

\global\long\def\V{H_{0}^{1}(\Omega;\C)}

\global\long\def\Vp{H^{-1}(\Omega;\C)}

\global\long\def\Vr{H_{0}^{1}(\Omega)}

\global\long\def\Honer{H^{1}(\Omega)}

\global\long\def\Hone{H^{1}(\Omega;\C)}

\global\long\def\Hhalf{H^{1/2}(\Omega;\C)}

\global\long\def\Hhalfr{H^{1/2}(\Omega)}

\global\long\def\phi{\varphi}

\global\long\def\epsilon{\varepsilon}

\global\long\def\div{{\rm div}}

\global\long\def\ld{L^{2}(\Omega;\C^{3})}

\global\long\def\ldr{L^{2}(\Omega;\R^{3})}

\global\long\def\linf{L^{\infty}(\Omega;\C)}

\global\long\def\linfr{L^{\infty}(\Omega;\R)}

\global\long\def\Conealp{\mathcal{C}^{1,\alpha}(\overline{\Omega};\C)}

\global\long\def\Conealpr{\mathcal{C}^{1,\alpha}(\overline{\Omega})}

\global\long\def\Calpr{\mathcal{C}^{0,\alpha}(\overline{\Omega};\R^{d\times d})}

\global\long\def\Calprsca{\mathcal{C}^{0,\alpha}(\overline{\Omega})}

\global\long\def\Czeroonescal{\mathcal{C}^{0,1}(\overline{\Omega})}

\global\long\def\Calpvect{\mathcal{C}^{0,\alpha}(\overline{\Omega};\R^{d})}

\global\long\def\Cone{\Cl^{1}(\overline{\Omega};\C)}

\global\long\def\Coner{\Cl^{1}(\overline{\Omega})}

\global\long\def\Kad{K_{ad}}

\global\long\def\mina{\lambda}

\global\long\def\maxa{\Lambda}

\global\long\def\Hcurl{H(\curl,\Omega)}

\global\long\def\Hmu{H^{\mu}(\curl,\Omega)}

\global\long\def\Hocurl{H_{0}(\curl,\Omega)}

\global\long\def\Hdiv{H(\div,\Omega)}

\global\long\def\k{\omega}

\global\long\def\div{{\rm div}}

\global\long\def\curl{{\rm curl}}

\global\long\def\supp{{\rm supp}}

\global\long\def\sp{{\rm span}}

\global\long\def\ii{\mathbf{i}}

\global\long\def\E{E_{\k}}

\global\long\def\Ei{E_{\k}^{i}}

\global\long\def\Hi{H_{\k}^{i}}

\global\long\def\H{H_{\k}}

\global\long\def\EE{\tilde{E}_{\k}}

\global\long\def\bo{\partial\Omega}

\global\long\def\Co{\Cl(\overline{\Omega};\C^{3})}

\global\long\def\ve{\theta}

\global\long\def\so{\hat{\sigma}}

\global\long\def\order{\kappa}

\global\long\def\p{C}

\global\long\def\pert{t}

\global\long\def\e{\mathbf{e}}

\global\long\def\epsi{\varepsilon}

\global\long\def\A{\mathcal{A}}

\title[Enforcing non-zero constraints in PDEs and applications ]{Enforcing local non-zero constraints in PDEs and applications to
hybrid imaging problems }

\author{Giovanni S. Alberti}

\address{Department of Mathematics and Applications, \'Ecole Normale Supérieure, 45 rue d'Ulm, 75005 Paris, France}

\email{giovanni.alberti@ens.fr}

\date{May 9, 2015}
\begin{abstract}
We study the boundary control of solutions of the Helmholtz and Maxwell
equations to enforce local non-zero constraints. These constraints
may represent the local absence of nodal or critical points, or that
certain functionals depending on the solutions of the PDE do not vanish
locally inside the domain. Suitable boundary conditions are classically
determined by using complex geometric optics solutions. This work
focuses on an alternative approach to this issue based on the use
of multiple frequencies. Simple boundary conditions and a finite number
of frequencies are explicitly constructed independently of the coefficients
of the PDE so that the corresponding solutions satisfy the required
constraints. This theory finds applications in several hybrid imaging
modalities: some examples are discussed.
\end{abstract}



\maketitle

\section{\label{sec:Introduction}Introduction}

The boundary control of the partial differential equation
\begin{equation}
\left\{ \begin{array}{l}
-\div(a\,\nabla u_{\k}^{i})-(\k^{2}\epsi+\ii\k\sigma)u_{\k}^{i}=0\qquad\text{in \ensuremath{\Omega},}\\
u_{\k}^{i}=\phi_{i}\qquad\text{on \ensuremath{\partial\Omega},}
\end{array}\right.\label{eq:helmholtz-intro-real}
\end{equation}
to enforce local non-zero constraints is the main topic of this work,
where $\Omega\subseteq\R^{d}$ is a smooth bounded domain, $a\in L^{\infty}(\Omega;\R^{d\times d})$
is a uniformly elliptic symmetric tensor and $\epsilon,\sigma\in L^{\infty}(\Omega;\R)$
satisfy $\epsilon>0$ and $\sigma\ge0$. More precisely, we want to
find suitable $\phi_{i}$'s such that the corresponding solutions
to \eqref{eq:helmholtz-intro-real} satisfy certain non-zero constraints
in $\Omega$. For example, we may look for $d+1$ boundary conditions
$\phi_{1},\dots,\phi_{d+1}$ such that, at least locally
\begin{equation}
\left|u_{\k}^{1}\right|\ge\p,\;\;\left|\det\begin{bmatrix}\nabla u_{\k}^{2} & \cdots & \nabla u_{\k}^{d+1}\end{bmatrix}\right|\ge\p,\;\;\bigl|\det\begin{bmatrix}u_{\k}^{1} & \cdots & u_{\k}^{d+1}\\
\nabla u_{\k}^{1} & \cdots & \nabla u_{\k}^{d+1}
\end{bmatrix}\bigr|\ge\p\label{eq:intro_1}
\end{equation}
for some $C>0$ or, more generally, for $b$ boundary values $\phi_{1},\dots,\phi_{b}$
such that the corresponding solutions verify $r$ conditions given
by
\begin{equation}
\bigl|\zeta^{j}\bigl(u_{\k}^{1},\dots,u_{\k}^{b}\bigr)\bigr|\ge C,\qquad j=1,\dots,r,\label{eq:intro_2}
\end{equation}
where the maps $\zeta^{j}$ depend on $u_{\k}^{i}$ and their derivatives.
Determinant constraints are very common in elasticity theory. As discussed
below, our motivation comes from  several hybrid imaging techniques
\cite{bal2012_review}. 

The problem of constructing such boundary conditions is usually set
for a fixed frequency $\k>0$. The classical way to tackle this problem
is by means of the so called complex geometric optics solutions. Introduced
by Sylvester and Uhlmann \cite{sylv1987}, CGO solutions are particular
highly oscillatory solutions of the Helmholtz equation \eqref{eq:helmholtz-intro-real}
in $\R^{d}$ such that for $t\gg1$ ($a=1$, $d=2$)
\[
u^{(t)}(x)\approx e^{tx_{1}}\left(\cos(tx_{2})+\ii\sin(tx_{2})\right)\quad\text{in }\Cl^{1}(\overline{\Omega};\C),
\]
and can be used to determine suitable illuminations by using the estimates
proved by Bal and Uhlmann \cite{bal2010inverse} (see also \cite{bal2012inversediffusion,bal2012_review,ammari2012quantitative}).
For example, setting $\phi_{1}\approx u_{|\bo}^{(t)}$, $\phi_{2}\approx\Re u_{|\bo}^{(t)}$
and $\phi_{3}\approx\Im u_{|\bo}^{(t)}$ gives an open set of boundary
conditions whose solutions satisfy the first two constraints of \eqref{eq:intro_1}.
Thus, CGO solutions represent a very important theoretical tool, but
have several drawbacks. First, the suitable $\phi_{i}$'s can only
be constructed when the parameters are smooth. Second, since $t\gg1$,
the exponential decay in the first variable gives small lower bounds
$C$ and the high oscillations make this approach hardly implementable.
Further, the construction depends on the coefficients $a$, $\epsilon$
and $\sigma$, that are usually unknown in inverse problems. Another
construction method uses the Runge approximation, which ensures that
locally the solutions behave as in the constant coefficient case \cite{bal2011reconstruction}.

In \cite{alberti2013multiple}, where the case $\sigma=0$ and the
constraints in \eqref{eq:intro_1} were considered, we proposed an
alternative approach to this issue based on the use of multiple frequencies
in a fixed admissible range $\A=[K_{\text{min}},K_{\text{max}}]\subseteq\R_{+}$.
The technique relies upon the assumption that the $\phi_{i}$'s are
chosen in such a way that the required constraints are satisfied in
the case $\k=0$, i.e. for the conductivity equation
\[
\left\{ \begin{array}{l}
-\div(a\,\nabla u_{0}^{i})=0\qquad\text{in \ensuremath{\Omega},}\\
u_{0}^{i}=\phi_{i}\qquad\text{on \ensuremath{\partial\Omega},}
\end{array}\right.
\]
for which the maximum principle and results on the absence of critical
points \cite{alessandrinimagnanini1994,bal-courdurier-2013} usually
make the problem much easier. Under this assumption, there exist a
finite $K\subseteq\A$ and an open cover $\Omega=\cup_{\k\in K}\Omega_{\k}$
such that the constraints are satisfied in each $\Omega_{\k}$ for
$u_{\k}^{i}$. The proof is based on the regularity theory and on
the holomorphicity of the map $\k\mapsto u_{\k}^{i}$.

The main novelty of this paper lies in the fully constructive proof.
The set $K$ is constructed explicitly as a uniform sampling of the
admissible range $\A$ and depends only on the a priori data. Similarly,
the constant $C$ in \eqref{eq:intro_2} is estimated a priori and
depend on the coefficients only through the a priori bounds. This
improvement has been achieved by using a quantitative version of the
unique continuation theorem for holomorphic functions proved by Momm
\cite{momm-1990} and a thorough analysis of \eqref{eq:helmholtz-intro-real}.
We consider here the case $\sigma\ge0$ and the general constraints
 \eqref{eq:intro_2}.

It is natural to study this issue for the full Maxwell's equations,
for which the Helmholtz equation often acts as an approximation in
the context of hybrid imaging. Maxwell's equations read
\begin{equation}
\left\{ \begin{array}{l}
\curl E_{\k}^{i}=\ii\k\mu H_{\k}^{i}\qquad\mbox{ in }\Omega,\\
\curl H_{\k}^{i}=-\ii(\k\epsi+\ii\sigma)E_{\k}^{i}\qquad\mbox{ in }\Omega,\\
E_{\k}^{i}\times\nu=\phi_{i}\times\nu\mbox{ on }\partial\Omega.
\end{array}\right.\label{eq:maxwell-intro}
\end{equation}
As before, we look for illuminations $\phi_{i}$ and frequencies $\k$
such that the corresponding solutions verify $r$ conditions given
by 
\begin{equation}
\bigl|\zeta^{j}\bigl((E_{\k}^{1},H_{\k}^{1}),\dots,(E_{\k}^{b},H_{\k}^{b})\bigr)\bigr|\ge C>0,\qquad j=1,\dots,r.\label{eq:zeta intro-maxwell}
\end{equation}
An example of such conditions is given by $\bigl|\det\begin{bmatrix}E_{\k}^{1} & E_{\k}^{2} & E_{\k}^{3}\end{bmatrix}\bigr|\ge C$.
CGO solutions for Maxwell's equations have been studied by Colton
and Päivärinta \cite{colton-paivarinta-1992}. As before, they can
be used to obtain suitable solutions \cite{chen-yang-2013}, but have
the drawbacks discussed before. In \cite{albertigsII}, the multi
frequency approach was generalised to \eqref{eq:maxwell-intro}. The
contribution of this paper is in the quantitative estimates for the
number of needed frequencies and for the constant $C$ in \eqref{eq:zeta intro-maxwell},
both determined a priori.

This approach has been recently successfully
adapted to the conductivity equation with complex coefficients  in \cite{ammari2014admittivity} 
and to the Helmholtz equation with Robin boundary conditions in \cite{alberti-ammari-ruan-2014}.

This theory finds applications in several hybrid imaging inverse problems,
where the unknown parameters have to be reconstructed from internal
data \cite{kuchment2011_review,bal2012_review,alberti-capdeboscq-2014,alberti-capdeboscq-2015}. Many hybrid problems
are governed by the Helmholtz equation \eqref{eq:helmholtz-intro-real},
e.g. microwave imaging by ultrasound deformation \cite{triki2010,cap2011},
quantitative thermo-acoustic \cite{bal2011quantitative,ammari2012quantitative},
transient elastography and magnetic resonance elastography \cite{bal2011reconstruction}.
The internal measurements are always linear or quadratic functionals
of $u_{\k}^{\phi}$ and of $\nabla u_{\k}^{\phi}$. For example, in
microwave imaging by ultrasound deformation, that is modelled by \eqref{eq:helmholtz-intro-real}
with a scalar-valued $a$ and $\sigma=0$, the internal measurements
have the form
\[
a(x)\left|\nabla u_{\k}^{\phi}\right|^{2}(x),\qquad\epsilon(x)\left|u_{\k}^{\phi}\right|(x)^{2},\qquad x\in\Omega,
\]
and in thermo-acoustic, modelled by \eqref{eq:helmholtz-intro-real}
with $a=\epsilon=1$ and $\sigma>0$, we measure
\[
\sigma(x)\left|u_{\k}^{\phi}\right|(x)^{2},\qquad x\in\Omega.
\]
In order for these measurements to be meaningful at every $x\in\Omega$,
they need to be non-zero: otherwise, we would measure only noise.
Moreover, we shall see that conditions like \eqref{eq:intro_1} or,
more generally, \eqref{eq:intro_2} for some map $\zeta$, are necessary
to reconstruct the unknown parameters $a$, $\epsilon$ and/or $\sigma$
or to obtain good stability estimates \cite{triki2010,kuchment2012stabilizing,bal2011reconstruction}.
Thus, being able to determine suitable illuminations independently
of the unknown parameters is fundamental, and these can be given by
the multi-frequency approach discussed in this paper. It should be
mentioned that stability of Hölder type has been proved by Alessandrini
in the context of microwave imaging with ultrasounds with $a=1$ without
requiring any non-zero constraint \cite{alessandrini2014global}.

Similarly, several problems are modelled by the Maxwell's equations
\eqref{eq:maxwell-intro} \cite{bal2013reconstruction,bal2013hybrid,chen-yang-2013},
and the inversion usually requires the availability of solutions satisfying
certain non-zero constraints inside the domain, given by \eqref{eq:zeta intro-maxwell},
for some maps $\zeta^{j}$ depending on the particular problem under
consideration. As above, the multi-frequency approach discussed in
this work can be applied to all these situations.

It is worth mentioning that the underlying physical principle was employed by Renzhiglova
et al. in an experimental study on magneto-acousto-eletrical tomography, where
dual-frequency ultrasounds were used to obtain non-zero internal data \cite{Renzhiglova}.

This paper is structured as follows. The main results are stated and
commented in Section~\ref{sec:Main-results}, and their proofs are
detailed in Section~\ref{sec:Proof-of-the}. Several applications
to hybrid imaging problems are described in Section~\ref{sec:Applications-to-hybrid}.
Some relevant open problems are discussed in Section~\ref{sec:Conclusions}.
Finally, some  basic tools are presented in Appendix~\ref{sec:Some-basic-tools}.

\section{\label{sec:Main-results}Main results}

\subsection{The Helmholtz equation}

Given a smooth bounded domain $\Omega\subseteq\R^{d}$, $d=2,3$,
we consider the Dirichlet boundary value problem
\begin{equation}
\left\{ \begin{array}{l}
-\div(a\,\nabla u_{\k}^{i})-(\k^{2}\epsi+\ii\k\sigma)\, u_{\k}^{i}=0\qquad\text{in \ensuremath{\Omega},}\\
u_{\k}^{i}=\phi_{i}\qquad\text{on \ensuremath{\partial\Omega}.}
\end{array}\right.\label{eq:helmholtz-multi}
\end{equation}
We assume that $a\in L^{\infty}(\Omega;\R^{d\times d})$ and $\epsi\in\linfr$
and satisfy\begin{subequations}\label{eq:ellipticity_a-epsi-multi}
\begin{align}
 & a=a^{T},\qquad\maxa^{-1}\left|\xi\right|^{2}\le\xi\cdot a\xi\le\maxa\left|\xi\right|^{2},\qquad\xi\in\R^{d},\label{eq:ellipticity_a-multi}\\
 & \maxa^{-1}\le\epsi\le\maxa\;\text{ almost everywhere}\label{eq:bounds_epsilon-multi}
\end{align}
\end{subequations}for some $\maxa>0$ and that $\sigma\in\linfr$
and satisfies either
\begin{align}
 & \sigma=0,\text{ or}\label{eq:sigma equal 0-multi}\\
 & \maxa^{-1}\le\sigma\le\maxa\;\text{ almost everywhere.}\label{eq:sigma bounds below-multi}
\end{align}
In electromagnetics, $\varepsilon$ is the electric permittivity,
$\sigma$ is the electric conductivity and $a$ is the inverse of
the magnetic permeability. \textcolor{black}{Take $\order\in\N$ and $\alpha\in(0,1)$. Suppose $\phi_{i}\in\Cl^{\order,\alpha}(\overline{\Omega};\C)$
and
\begin{equation}
a\in\Cl^{\order-1,\alpha}(\overline{\Omega};\R^{d\times d}),\quad\epsi,\sigma\in W^{\order-1,\infty}(\Omega;\R)\quad\text{if $\order\ge 1$.}\label{eq:regularity assumption-helmh-multi}
\end{equation}
}

Let $\A=[K_{min},K_{max}]\subseteq B(0,M)$ represent the frequencies
we have access to, for some $0<K_{min}<K_{max}\le M$. By standard
elliptic theory (Proposition~\ref{prop:helmoltz-wellposedness}),
problem \eqref{eq:helmholtz-multi} is well-posed for every $\k\in D$,
where
\begin{equation}
D=\begin{cases}
\C\setminus\sqrt{\Sigma} & \text{if \eqref{eq:sigma equal 0-multi} holds,}\\
\{\k\in\C:\left|\Im\k\right|<\eta\} & \text{if \eqref{eq:sigma bounds below-multi} holds.}
\end{cases}\label{eq:definition of D}
\end{equation}
Here $\Sigma=\{\lambda_{l}:l\in\N^{*}\}$ is the set of the Dirichlet
eigenvalues of \textcolor{black}{problem \eqref{eq:helmholtz-multi}} ($\sqrt{\Sigma}=\{\k\in\C:\k^{2}\in\Sigma\}$),
and $\eta>0$ depends only on $\Omega$ and $\maxa$. Figure~\vref{fig:domain D}
represents the domain $D$ and the admissible set of frequencies $\A$.
Note that $u_{\k}^{i}\in\Cl^{\order}(\overline{\Omega};\C)$ by elliptic
regularity theory (Proposition~\ref{pro: Helmholtz regularity}).
\begin{defn}
Given a finite set $K\subseteq\A$ and $\phi_{1},\dots,\phi_{b}\in\Cl^{\order,\alpha}(\overline{\Omega};\C)$,
we say that $K\times\{\phi_{1},\dots,\phi_{b}\}$ is a \emph{set of
measurements\index{set of measurements (Helmholtz)}}.
\end{defn}
We shall study a particular class of sets of measurements, namely
those whose corresponding solutions $u_{\k}^{i}$ ($i=1,\dots,b$)
to \eqref{eq:helmholtz-multi} and their derivatives up to the $\kappa$-th
order satisfy $r$  constraints in $\Omega$. These are described
by a map $\zeta$. For $b,r\in\N^{*}$ let\begin{subequations}\label{eq:assumptions map zed}
\begin{align}
 & \zeta=(\zeta^{1},\dots,\zeta^{r})\colon\Cl^{\order}(\overline{\Omega};\C){}^{b}\longrightarrow\Cl(\overline{\Omega};\C)^{r}\;\text{ be holomorphic, such that}\label{eq:assumptions map zed-def}\\
 & \bigl\Vert\zeta(u^{1},\dots,u^{b})\bigr\Vert_{\Cl(\overline{\Omega};\C)^{r}}\le c_{\zeta}(1+\bigl\Vert(u^{1},\dots,u^{b})\bigr\Vert_{\Cl^{\order}(\overline{\Omega};\C){}^{b}}^{s})\text{ and}\label{eq:assumptions map zed-a}\\
 & \bigl\Vert D\zeta_{(u^{1},\dots,u^{b})}\bigr\Vert_{\mathcal{B}(\Cl^{\order}(\overline{\Omega};\C){}^{b},\Cl(\overline{\Omega};\C)^{r})}\le c_{\zeta}(1+\bigl\Vert(u^{1},\dots,u^{b})\bigr\Vert_{\Cl^{\order}(\overline{\Omega};\C){}^{b}}^{s})\label{eq:assumptions map zed-b}
\end{align}
\end{subequations}for some $c_{\zeta}>0$ and $s\in\N^{*}$. \textcolor{black}{(For the definition of holomorphic function, see  $\S$~\ref{sub:Preliminary-lemmata}.)} We shall
use the notation $C_{\zeta}=(c_{\zeta},s,r,\order,\alpha)$.
\begin{example}
\label{exa:zeta_det}We consider here the constraints given in \eqref{eq:intro_1}.
Take $b=d+1$, $r=3$ and $\order=1$ and let $\zeta_{\det}\colon\Cl^{1}(\overline{\Omega};\C){}^{d+1}\longrightarrow\Cl(\overline{\Omega};\C)^{3}$
be defined by 
\begin{align*}
 & \zeta_{\det}^{1}(u^{1},\dots,u^{d+1})=u^{1},\\
 & \zeta_{\det}^{2}(u^{1},\dots,u^{d+1})=\det\begin{bmatrix}\nabla u^{2} & \cdots & \nabla u^{d+1}\end{bmatrix},\\
 & \zeta_{\det}^{3}(u^{1},\dots,u^{d+1})=\det\begin{bmatrix}u^{1} & \cdots & u^{d+1}\\
\nabla u^{1} & \cdots & \nabla u^{d+1}
\end{bmatrix}.
\end{align*}
The map $\zeta_{\det}$ is holomorphic (Lemma~\ref{lem:analytic functions}).
Simple calculations show that \eqref{eq:assumptions map zed-a} holds
true with $s_{b}=d+1$ and \eqref{eq:assumptions map zed-b} with
$s_{c}=d$, and so we can set $s=d+1$.
\end{example}
We introduce the particular class of sets of measurements we are interested
in.
\begin{defn}
\label{def:zeta-complete}Take $\Omega'\subseteq\Omega$. Let $b,r\in\N^{*}$
be two positive integers, $C>0$ and let $\zeta$ be as in \eqref{eq:assumptions map zed}.
A set of measurements $K\times\{\phi_{1},\dots,\phi_{b}\}$ is \emph{$(\zeta,\p)$-complete
in $\Omega'$} if there exists an open cover of $\Omega'$
\[
\Omega'=\bigcup_{\k\in K\cap D}\Omega'_{\k},
\]
such that for any $\k\in K\cap D$ 
\begin{equation}
\bigl|\zeta^{j}\bigl(u_{\k}^{1},\dots,u_{\k}^{b}\bigr)(x)\bigr|\ge\p,\qquad j=1,\dots,r,\; x\in\Omega'_{\k}.\label{eq:zeta-complete}
\end{equation}

\end{defn}
Namely, a $(\zeta,C)$-complete set gives a cover of $\Omega'$ into
$\#(K\cap D)$ subdomains, such that the constraints given in \eqref{eq:zeta-complete}
are satisfied in each subdomain for different frequencies.

We now describe how to choose the frequencies in the admissible set
$\A$. Let $K^{(n)}$\index{K@$K^{(n)}$} be the uniform partition
of $\A$ into $n-1$ intervals so that $\#K^{(n)}=n$, namely
\begin{equation}
K^{(n)}=\{\k_{1}^{(n)},\dots,\k_{n}^{(n)}\},\qquad\k_{i}^{(n)}=K_{min}+\frac{(i-1)}{(n-1)}(K_{max}-K_{min}).\label{eq:K^(n)}
\end{equation}
\textcolor{black}{Set $\left|\A\right|=K_{max}-K_{min}$.} The main result of this paper regarding the Helmholtz equation reads
as follows.
\begin{thm}
\label{thm:quantitative bounds-helmholtz-zeta}Assume that \eqref{eq:ellipticity_a-epsi-multi},
\eqref{eq:regularity assumption-helmh-multi} and either \eqref{eq:sigma equal 0-multi}
or \eqref{eq:sigma bounds below-multi} hold. Let $\zeta$ be as in
\eqref{eq:assumptions map zed} and \textcolor{black}{ assume that there exist $\phi_{1},\dots,\phi_{b}\in\Cl^{\order,\alpha}(\overline{\Omega};\C)$ and  $C_{0}>0$ such that
\begin{equation}
\bigl|\zeta^{j}\bigl(u_{0}^{1},\dots,u_{0}^{b}\bigr)(x)\bigr|\ge\p_{0},\quad j=1,\dots,r,\; x\in\Omega'.\label{eq:assumption k=00003D0}
\end{equation}
}Then there exist $C>0$ and $n\in\N$ depending
on $\Omega$,  $\maxa$, $\left|\A\right|$, $M$, $C_{\zeta}$, $\left\Vert a\right\Vert _{\Cl^{\order-1,\alpha}(\overline{\Omega};\R^{d\times d})}$,
$\left\Vert (\epsilon,\sigma)\right\Vert _{W^{\order-1,\infty}(\Omega;\R)^{2}}$,
$\left\Vert \phi_{i}\right\Vert _{\Cl^{\order,\alpha}(\overline{\Omega};\C)}$
and $C_{0}$ such that
\[
K^{(n)}\times\{\phi_{1},\dots,\phi_{b}\}
\]
is a $(\zeta,C)$-complete set of measurements in $\Omega'$.
\end{thm}
We now discuss assumption \eqref{eq:assumption k=00003D0}, the dependence
of $C$ on $\left|\A\right|$ and $M$ and the regularity assumption
on the coefficients.
\begin{rem}
\label{rem:dependence-helmholtz}This result allows an a priori construction
of $(\zeta,C)$-complete sets, since $C$ and $n$ depend only on
a priori data, provided that $\phi_{1},\dots,\phi_{b}$ are chosen
in such a way that \eqref{eq:assumption k=00003D0} holds true. It
is in general easier to satisfy \eqref{eq:assumption k=00003D0} than
\eqref{eq:zeta-complete}, as $\k=0$ makes problem \eqref{eq:helmholtz-multi}
simpler. More precisely, there exist many results regarding the conductivity
equation \cite{alessandrinimagnanini1994,cap2009,widlak2012hybrid,bal-courdurier-2013,bal2012inversediffusion}
(see also the proof of Corollary~\ref{cor:det-complete}). \textcolor{black}{It is worth noting that, especially in 3D, satisfying \eqref{eq:assumption k=00003D0} may still be highly non trivial, and  the strategy used for the case $\k=0$ may be applicable for higher frequencies as well.}

Note that \eqref{eq:helmholtz-multi} with $\k=0$ does not depend
on $\epsilon$ and $\sigma$, so that the construction of $\phi_{1},\dots,\phi_{b}$
is always independent of $\epsilon$ and $\sigma$ but may depend
on $a$. 

There exist \emph{occulting} illuminations, i.e. boundary conditions
for which a finite number of frequencies are not sufficient, and so
assumption \eqref{eq:assumption k=00003D0} cannot be completely removed
\cite{alberti2013multiple}. Yet, this assumption can be weakened
(see Remark~\ref{rem:generic boundary}).
\end{rem}

\begin{rem}
The proof of this result is based on Lemma~\ref{lem:momm-ellipse}.
Thus, the constant $C$ goes to zero as
$\left|\A\right|\to0$, $M\to\infty$ \textcolor{black}{or $C_0\to 0$} \textcolor{black}{(see Remark~\ref{rem:momm} for the precise dependence)}. In particular, this approach
gives good estimates for frequencies in a moderate regime (e.g. with
microwaves), but these estimates get worse for very high frequencies. \textcolor{black}{This should be taken into account in the presence of noisy measurements.}
\end{rem}

\begin{rem}
\label{rem:regularity of coefficients-helmholtz}The regularity of
the coefficients required for this approach is lower than the regularity
required if CGO solutions are used. Indeed, consider for simplicity
the constraints given by the map $\zeta_{\det}$ and suppose $a=1$
and $\sigma=0$. The CGO approach requires $\epsilon\in\Cl^{1}$ \cite{bal2010inverse},
while with this method we only assume $\epsilon\in L^{\infty}$.

\textcolor{black}{Similarly, the approach based on the Runge approximation property requires $a$ to be Lipschitz continuous, in addition to \eqref{eq:regularity assumption-helmh-multi} \cite{bal2011reconstruction}. Therefore, higher regularity assumptions are needed in the cases when $\order=0,1$.}
\end{rem}
We now apply Theorem~\ref{thm:quantitative bounds-helmholtz-zeta}
to the case $\zeta=\zeta_{\det}$. The construction of $(\zeta_{\det},C)$-complete
sets of measurements depends on the dimension, since the validity
of \eqref{eq:assumption k=00003D0} for $\zeta_{\det}^{2}$ and $\zeta_{\det}^{3}$
depends on the dimension.
\begin{cor}
\label{cor:det-complete}Assume that \eqref{eq:ellipticity_a-epsi-multi},
\eqref{eq:regularity assumption-helmh-multi} and either \eqref{eq:sigma equal 0-multi}
or \eqref{eq:sigma bounds below-multi} hold for $\order=1$. 

If $d=2$, $\Omega$ is convex and $\Omega'\Subset\Omega$ then there
exist $C>0$ and $n\in\N$ depending on $\Omega$, $\Omega'$, $\maxa$,
$\alpha$, $\left|\A\right|$, $M$ and $\left\Vert a\right\Vert _{\Cl^{0,\alpha}(\overline{\Omega};\R^{2\times2})}$
such that
\[
K^{(n)}\times\{1,x_{1},x_{2}\}
\]
is a $(\zeta_{\det},C)$-complete set of measurements in $\Omega'$.

If $d=3$ and $\hat{a}\in\R^{3\times3}$ \textcolor{black}{is a constant tensor satisfying} \eqref{eq:ellipticity_a-multi}
then there exist $\delta,C>0$ and $n\in\N$ depending on $\Omega$,
 $\maxa$, $\alpha$, $\left|\A\right|$, $M$ and $\left\Vert a\right\Vert _{\Cl^{0,\alpha}(\overline{\Omega};\R^{3\times3})}$
such that if $\left\Vert a-\hat{a}\right\Vert _{\Cl^{0,\alpha}(\overline{\Omega};\R^{3\times3})}\le\delta$
then 
\[
K^{(n)}\times\{1,x_{1},x_{2},x_{3}\}
\]
is a $(\zeta_{\det},C)$-complete set of measurements in $\Omega$.\end{cor}
\begin{rem}
In 2D, it is possible to consider non-convex domains, provided that
the boundary conditions are chosen in accordance to Lemma~\ref{lem:ales-magn}
\cite{bauman2001univalent,alberti2013multiple}.
\end{rem}

\begin{rem}
In order to satisfy the constraints corresponding to $\zeta_{\det}^{1}$,
by the strong maximum principle it is enough to choose $\phi_{1}\ge C_{0}>0$.
As far as \eqref{eq:assumption k=00003D0} for $\zeta_{\det}^{3}$
is concerned, it is sufficient to set $\phi_{2}=x_{1}\phi_{1}$ and
$\phi_{3}=x_{2}\phi_{1}$ \cite{alberti2013multiple}.
\end{rem}

\begin{rem}
\label{rem:generic boundary}The difference between the two and three
dimensional case is due to the presence of critical points in the
case $\k=0$ in 3D \cite{briane-milton-nesi-2004,bal2013cauchy,capdeboscq-2015}.
In order to satisfy \eqref{eq:assumption k=00003D0} in 3D we assume
that $a$ is close to a constant matrix. This assumption can be removed
in some situations by using a different approach in $\k=0$ \cite{bal-courdurier-2013}
or by choosing generic boundary conditions \cite{alberti-genericity}: in
these cases, the a priori estimates on $C$ and $n$ are lost. If
the constraints do not involve gradient fields, e.g. $\zeta=\zeta_{\det}^{1}$,
then there is no need for this assumption.
\end{rem}

\subsection{Maxwell's equations}

Given a smooth bounded domain $\Omega\subseteq\R^{3}$ with a simply
connected boundary $\bo$, in this subsection we consider Maxwell's
equations
\begin{equation}
\left\{ \begin{array}{l}
\curl\Ei=\ii\k\mu\Hi\qquad\text{in \ensuremath{\Omega},}\\
\curl\Hi=-\ii(\k\epsi+\ii\sigma)\Ei\qquad\text{in \ensuremath{\Omega},}\\
\Ei\times\nu=\phi_{i}\times\nu\qquad\text{on \ensuremath{\partial\Omega},}
\end{array}\right.\label{eq:combined i-maxwell}
\end{equation}
with $\mu,\epsi,\sigma\in L^{\infty}(\Omega;\R^{3\times3})$ and $\phi_{i}$
satisfying\begin{subequations}\label{eq:assumption_coefficients-maxwell-multi}
\begin{align}
 & \maxa^{-1}\left|\xi\right|^{2}\le\xi\cdot\mu\xi,\quad\maxa^{-1}\left|\xi\right|^{2}\le\xi\cdot\epsilon\xi,\quad\maxa^{-1}\left|\xi\right|^{2}\le\xi\cdot\sigma\xi,\qquad\xi\in\R^{3},\label{eq:assumption_ell-maxwell-multi}\\
 & \left\Vert (\mu,\epsilon,\sigma)\right\Vert _{L^{\infty}(\Omega;\R^{3\times3})^{3}}\le\maxa,\;\mu=\mu^{T},\;\epsi=\epsi^{T},\;\sigma=\sigma^{T},\;\mu,\epsilon,\sigma\in W^{\order+1,p}(\Omega),\label{eq:assumption_bounds-maxwell-multi}\\
 & \curl\phi_{i}\cdot\nu=0\text{ on \ensuremath{\bo}}\text{ and }\phi_{i}\in W^{\order+1,p}(\Omega;\C^{3})\label{eq:assumption_phi-maxwell-multi}
\end{align}
\end{subequations}for some $\maxa>0$, $\order\in\N$ and $p>3$.
The electromagnetic fields $E_{\k}^{i}$ and $H_{\k}^{i}$ satisfy
\begin{align*}
 & E_{\k}^{i}\in\Hcurl:=\{u\in\ld:\curl u\in\ld\},\\
 & H_{\k}^{i}\in\Hmu:=\{v\in\Hcurl:\div(\mu v)=0\text{ in \ensuremath{\Omega}, }\mu v\cdot\nu=0\text{ on }\bo\}.
\end{align*}
The matrix $\varepsilon$ represents the electric permittivity, $\sigma$
is the electric conductivity and $\mu$ stands for the  magnetic permeability.
Note that $(E_{\k}^{i},H_{\k}^{i})\in\Cl^{\order}(\overline{\Omega};\C^{6})$
by Proposition~\ref{prop:regularity-maxwell}.
\begin{defn}
Given a finite set $K\subseteq\A$ and $\phi_{1},\dots,\phi_{b}\in W^{\order+1,p}(\Omega;\C^{3})$
satisfying \eqref{eq:assumption_phi-maxwell-multi}, we say that $K\times\{\phi_{1},\dots,\phi_{b}\}$
is a \emph{set of measurements}.
\end{defn}
As before, we are interested in a particular class of sets of measurements,
namely those whose corresponding solutions $(E_{\k}^{i},H_{\k}^{i})$
to \eqref{eq:combined i-maxwell} and their derivatives up to the
$\kappa$-th order satisfy $r$ non-zero constraints inside the domain.
These are described by a map $\zeta$, which we now introduce. For
$b,r\in\N^{*}$ let\begin{subequations}\label{eq:assumptions map zed-maxwell}
\begin{align}
 & \zeta=(\zeta^{1},\dots,\zeta^{r})\colon\Cl^{\order}(\overline{\Omega};\C^{6}){}^{b}\longrightarrow\Cl(\overline{\Omega};\C)^{r}\;\text{ be holomorphic, such that}\label{eq:definition of zeta-maxwell}\\
 & \bigl\Vert\zeta((u^{i},v^{i})_{i})\bigr\Vert_{\Cl(\overline{\Omega};\C)^{r}}\le c_{\zeta}(1+\bigl\Vert((u^{i},v^{i})_{i})\bigr\Vert_{\Cl^{\order}(\overline{\Omega};\C^{6}){}^{b}}^{s}),\label{eq:assumptions map zed-maxwell-a}\\
 & \bigl\Vert D\zeta_{((u^{i},v^{i})_{i})}\bigr\Vert_{\mathcal{B}(\Cl^{\order}(\overline{\Omega};\C^{6}){}^{b},\Cl(\overline{\Omega};\C)^{r})}\!\le\! c_{\zeta}(1+\bigl\Vert((u^{i},v^{i})_{i})\bigr\Vert_{\Cl^{\order}(\overline{\Omega};\C^{6}){}^{b}}^{s})\label{eq:assumptions map zed-maxwell-b}
\end{align}
\end{subequations}for some $c_{\zeta}>0$ and $s\in\N^{*}$. We shall
use the notation $C_{\zeta}=(c_{\zeta},s,r,\order,p)$\index{C@$C_\zeta$ (Maxwell)}.

We now consider one example of map $\zeta$. For other examples, see
\cite{albertigsII}.
\begin{example}
\label{exa:zeta_det_maxwell}Take $b=3$, $r=1$, $\order=0$ and
let $\zeta_{\det}^{M}$ be defined by
\[
\zeta_{\det}^{M}((u_{1},v_{1}),(u_{2},v_{2}),(u_{3},v_{3}))=\det\begin{bmatrix}u_{1} & u_{2} & u_{3}\end{bmatrix},\qquad(u_{i},v_{i})\in\Cl(\overline{\Omega};\C^{6}).
\]
The map $\zeta_{\det}^{M}$ is multilinear and bounded, whence holomorphic
by Lemma~\ref{lem:analytic functions}. Assumptions \eqref{eq:assumptions map zed-maxwell-a}
and \eqref{eq:assumptions map zed-maxwell-b} are obviously verified.
In this case, the condition characterising $(\zeta_{\det}^{M},C)$-complete
sets of measurements is $\bigl|\det\begin{bmatrix}E_{\k}^{1} & E_{\k}^{2} & E_{\k}^{3}\end{bmatrix}(x)\bigr|\ge\p$.
In other words, this constraints signals the availability, in every
point, of three independent electric fields and, in particular, of
one non-vanishing electric field. 
\end{example}
We now give the precise definition of $(\zeta,C)$-complete sets of
measurements for Maxwell's equations. The only difference with the
Helmholtz equation is that here, for simplicity, we require the constraints
to hold in the whole domain $\Omega$.
\begin{defn}
\label{def:zeta-complete-maxwell}Let $b,r\in\N^{*}$ be two positive
integers, $C>0$ and let $\zeta$ be as in \eqref{eq:assumptions map zed-maxwell}.
A set of measurements $K\times\{\phi_{1},\dots,\phi_{b}\}$ is \emph{$(\zeta,\p)$-complete}
if there exists an open cover of $\Omega$, $\Omega=\cup_{\k\in K}\Omega{}_{\k}$,
such that for any $\k\in K$
\begin{equation}
\bigl|\zeta^{j}\bigl((E_{\k}^{1},H_{\k}^{1}),\dots,(E_{\k}^{b},H_{\k}^{b})\bigr)(x)\bigr|\ge\p,\qquad j=1,\dots,r,\; x\in\Omega{}_{\k}.\label{eq:zeta-complete-maxwell}
\end{equation}

\end{defn}
Let $K^{(n)}$ be as in \eqref{eq:K^(n)}. The main result of this
subsection reads as follows.
\begin{thm}
\label{thm:quantitative bounds-complex maxwell-zeta}Assume that \eqref{eq:assumption_coefficients-maxwell-multi}
holds. Let $\so\in W^{\order,p}(\Omega;\R^{3\times3})$ satisfy \eqref{eq:assumption_ell-maxwell-multi}.
Let $\zeta$ be as in \eqref{eq:assumptions map zed-maxwell} and \textcolor{black}{assume that there exist
$\phi_{1},\dots,\phi_{b}\in W^{\order+1,p}(\Omega;\C^{3})$ satisfying
\eqref{eq:assumption_phi-maxwell-multi} and $C_{0}>0$ such that
\begin{equation}
\bigl|\zeta^{j}\bigl((\hat{E}_{0}^{1},\hat{H}_{0}^{1}),\dots,(\hat{E}_{0}^{b},\hat{H}_{0}^{b})\bigr)(x)\bigr|\ge C_{0},\qquad x\in\Omega,\, j=1,\dots,r,\label{eq:assumption k 0}
\end{equation}
}where $(\hat{E}{}_{0}^{i},\hat{H}_{0}^{i})\in\Hcurl\times\Hmu$
is the solution to \eqref{eq:combined i-maxwell} with $\so$ in lieu
of \textup{$\sigma$} and $\k=0$, namely
\begin{equation}
\left\{ \begin{array}{l}
\curl\hat{E}{}_{0}^{i}=0\qquad\text{in \ensuremath{\Omega},}\\
\div(\so\hat{E}{}_{0}^{i})=0\qquad\text{in \ensuremath{\Omega},}\\
\hat{E}{}_{0}^{i}\times\nu=\phi_{i}\times\nu\qquad\text{on \ensuremath{\partial\Omega},}
\end{array}\right.\qquad\left\{ \begin{array}{l}
\curl\hat{H}_{0}^{i}=\so\hat{E}{}_{0}^{i}\qquad\text{in \ensuremath{\Omega},}\\
\div(\mu\hat{H}_{0}^{i})=0\qquad\text{in \ensuremath{\Omega},}\\
\mu\hat{H}_{0}^{i}\cdot\nu=0\qquad\text{on \ensuremath{\partial\Omega}.}
\end{array}\right.\label{eq: k  0}
\end{equation}
There exist $\delta,C>0$ and $n\in\N$ depending on $\Omega$,  $\maxa$,
$\left|\A\right|$, $M$, $C_{\zeta}$, $\left\Vert \phi_{i}\right\Vert _{W^{\order+1,p}(\overline{\Omega};\C^{3})}$,
$\left\Vert (\epsi,\sigma,\mu)\right\Vert _{W^{\order+1,p}(\overline{\Omega};\R^{3\times3})}$
and $C_{0}$ such that if $\left\Vert \sigma-\so\right\Vert _{W^{\order+1,p}(\overline{\Omega};\R^{3\times3})}\le\delta$
then 
\[
K^{(n)}\times\{\phi_{1},\dots,\phi_{b}\}
\]
is a $(\zeta,C)$-complete set of measurements.
\end{thm}
We now discuss assumption \eqref{eq:assumption k 0}, the dependence
of the construction of the illuminations on the electromagnetic parameters
and the regularity assumption on the coefficients (see Remarks~\ref{rem:dependence-helmholtz}
and \ref{rem:regularity of coefficients-helmholtz}).
\begin{rem}
\label{rem:dependence}Suppose that we are in the simpler case $\so=\sigma$.
Note that \eqref{eq: k  0} does not depend on $\epsilon$, so that
the construction of  $\phi_{1},\dots,\phi_{b}$ is always independent
of $\epsilon$ but may depend on $\sigma$ and $\mu$. However, in
the cases where the maps $\zeta^{j}$ involve only the electric field
$E$, it depends on $\sigma$, and not on $\epsilon$ and $\mu$ (see
Corollary~\ref{cor:zeta-complete}).

A typical application of the theorem is in the case where $\sigma$
is a small perturbation of a known constant tensor $\so$. Then, the
construction of $\phi_{1},\dots,\phi_{b}$ is independent of $\sigma$.
A similar argument would work if $\mu$ were a small perturbation of
a constant tensor $\hat{\mu}$. We have decided to omit it for simplicity,
since in the applications we have in mind the maps $\zeta^{j}$ do
not depend on the magnetic field $H$.
\end{rem}

\begin{rem}
\label{rem:regularity of coefficients}The regularity of the coefficients
required for this approach is much lower than the regularity required
if CGO solutions are used. Indeed, if the constraints depend on the
derivatives up to the $\kappa$-th order, with this approach we require
the parameters to be in $W^{\kappa+1,p}$, while with CGO we need
$W^{\kappa+3,p}$ \cite{chen-yang-2013}.
\end{rem}
In the case where the conditions given by the map $\zeta$ are independent
of the magnetic field $H$, Theorem~\ref{thm:quantitative bounds-complex maxwell-zeta}
can be rewritten in the following form.
\begin{cor}
\label{cor:zeta-complete}Assume that \eqref{eq:assumption_coefficients-maxwell-multi}
holds. Let $\so\in W^{\order,p}(\Omega;\R^{3\times3})$ satisfy \eqref{eq:assumption_ell-maxwell-multi}
and $\zeta$ be as in \eqref{eq:assumptions map zed-maxwell} and
independent of $H$. Take $\psi_{1},\dots,\psi_{b}\in W^{\kappa+2,p}(\Omega;\C)$.
Suppose
\begin{equation}
\bigl|\zeta^{j}\bigl(\nabla w^{1},\dots,\nabla w^{b}\bigr)(x)\bigr|\ge C_{0},\qquad x\in\Omega,\, j=1,\dots,r\label{eq:assumption k 0-cor}
\end{equation}
for some $C_{0}>0$, where $w^{i}\in\Hone$ is the solution to
\[
\left\{ \begin{array}{l}
\div(\so\nabla w^{i})=0\qquad\text{in \ensuremath{\Omega},}\\
w^{i}=\psi_{i}\qquad\text{on \ensuremath{\partial\Omega}.}
\end{array}\right.
\]
There exist $\delta,C>0$ and $n\in\N$ depending on $\Omega$,  $\maxa$,
$\left|\A\right|$, $M$, $C_{\zeta}$, $\left\Vert \psi_{i}\right\Vert _{W^{\order+2,p}(\overline{\Omega};\C^{3})}$,
$\left\Vert (\epsi,\sigma,\mu)\right\Vert _{W^{\order+1,p}(\overline{\Omega};\R^{3\times3})}$
and $C_{0}$ such that if $\left\Vert \sigma-\so\right\Vert _{W^{\order+1,p}(\overline{\Omega};\R^{3\times3})}\le\delta$
then 
\[
K^{(n)}\times\{\nabla\psi_{1},\dots,\nabla\psi_{b}\}
\]
is a $(\zeta,C)$-complete set of measurements.
\end{cor}
In other words, if the required constraints do not depend on $H$,
then the problem of finding $\zeta$-complete sets is reduced to satisfying
the same conditions for the gradients of solutions to the conductivity
equation, as with the Helmholtz equation.

\section{\label{sec:Proof-of-the}Non-zero constraints in PDEs}

The results stated in Section~\ref{sec:Main-results} are proven
here. In particular, some preliminary lemmata on holomorphic functions
are discussed in $\S$~\ref{sub:Preliminary-lemmata}, and the proofs
of Theorem~\ref{thm:quantitative bounds-helmholtz-zeta}, Corollary~\ref{cor:det-complete}
and Theorem~\ref{thm:quantitative bounds-complex maxwell-zeta} are
given in $\S$~\ref{sub:Proof-of-Theorem-1}, $\S$~\ref{sub:Proof-of-Corollary}
and $\S$~\ref{sub:Proof-of-Theorem-2}, respectively.

\subsection{\label{sub:Preliminary-lemmata}Holomorphic functions}

Holomorphic functions in a Banach space setting were studied in \cite{taylor37}.
Let $E$ and $E'$ be complex Banach spaces, $D\subseteq E$ be an
open set and take $f\colon D\to E'$. We say that $f$ is holomorphic
if it is continuous and if
\[
\lim_{\tau\to0}\frac{f(x_{0}+\tau y)-f(x_{0})}{\tau}
\]
exists in $E'$ for all $x_{0}\in D$ and $y\in E$. This notion extends
the classical notion of holomorphicity for functions of complex variable.

This lemma summarises some of the basic properties of holomorphic
functions. 
\begin{lem}
\label{lem:analytic functions}Let $E_{1},\dots,E_{r}$, $E$ and
$E'$ be complex Banach spaces and $D\subseteq E$ be an open set.
\begin{enumerate}
\item If $f\colon E_{1}\times\dots\times E_{r}\to E'$ is multilinear and
bounded then $f$ is holomorphic.
\item If $f\colon D\to E_{1}$ and $g\colon E_{1}\to E'$ are holomorphic
then $g\circ f\colon D\to E'$ is holomorphic.
\item Take $f=(f^{1},\dots f^{r})\colon D\to E_{1}\times\dots\times E_{r}$.
Then $f$ is holomorphic if and only if $f^{j}$ is holomorphic for
every $j=1,\dots,r$.
\end{enumerate}
\end{lem}
The following result is a quantitative version of the unique continuation
property for holomorphic functions of one complex variable.
\begin{lem}
\label{lem:momm}Take $C_{0},D>0$, $\theta\in(0,1)$ and $r\in(0,\theta]$.
Let $g$ be a holomorphic function in \textcolor{black}{$B(0,1)\subseteq\C$} such that $\left|g(0)\right|\ge C_{0}$
and $\sup_{B(0,1)}\left|g\right|\le D$. There exists $\k\in[r,1)$
such that
\[
\left|g(\k)\right|\ge C
\]
for some constant $C>0$ depending on $\theta$, $C_{0}$ and $D$
only.\end{lem}
\begin{proof}
Since $[\theta,(1+\theta)/2]\subseteq[r,1)$, it is sufficient to
show that there exists $C>0$ depending on $\theta$, $C_{0}$ and
$D$ only such that
\[
\max_{[\theta,(1+\theta)/2]}\left|g\right|\ge C.
\]
By contradiction, suppose that there exists a sequence $(g_{n})_{n}$
of holomorphic functions in $B(0,1)$ such that $\sup_{B(0,1)}\left|g_{n}\right|\le D$,
$\left|g_{n}(0)\right|\ge C_{0}$ and $\max_{[\theta,(1+\theta)/2]}\left|g_{n}\right|\to0$.
Since $\sup_{B(0,1)}\left|g_{n}\right|\le D$, by standard complex
analysis, up to a subsequence $g_{n}\to g_{\infty}$ for some $g_{\infty}$
holomorphic in $B(0,1)$. As $\max_{[\theta,(1+\theta)/2]}\left|g_{n}\right|\to0$,
we obtain $g_{\infty}=0$ on $[\theta,(1+\theta)/2]$, whence $g_{\infty}=0$,
which contradicts $\left|g_{\infty}(0)\right|\ge C_{0}$.\end{proof}
\begin{rem}
\label{rem:momm}Although elementary, the proof of Lemma~\ref{lem:momm}
does not give the dependence of the constant $C$ on the parameters
$\theta$, $C_{0}$ and $D$. 

By \cite{momm-1990} there is a Jordan curve $\Gamma$ in $r<\left|\k\right|<1$
around the origin such that
\[
\log\left|g(\k)/g(0)\right|\ge-\frac{\tilde{C}}{1-r}\biggl(\int_{0}^{1}\left(\frac{\log\sup_{B(0,t)}\left|g/g(0)\right|}{1-t}\right)^{1/2}dt\biggr)^{2},\qquad\k\in\Gamma,
\]
for an absolute constant $\tilde{C}>0$. By the Jordan curve theorem
there exists $\k\in(r,1)$ such that
\[
\log\left|g(\k)/g(0)\right|\ge-\frac{\tilde{C}\log(DC_{0}^{-1})}{1-r}.
\]
Therefore $\left|g(\k)\right|\ge\left|g(0)\right|(DC_{0}^{-1})^{-\frac{\tilde{C}}{1-r}}\ge C_{0}(DC_{0}^{-1})^{-\frac{\tilde{C}}{1-r}}\ge C_{0}(DC_{0}^{-1})^{-\frac{\tilde{C}}{1-\theta}}$,
whence the constant given in Lemma~\ref{lem:momm} is $C=C_{0}(DC_{0}^{-1})^{-\frac{\tilde{C}}{1-\theta}}$. 
\end{rem}
It is possible to generalise the previous result to functions defined
in an ellipse. The proof is elementary, but needed to show the precise
dependence of $C$ on $R_{1}-r$.
\begin{lem}
\label{lem:momm-ellipse}Take $0<r<R_{1}\le M$ and $0<\eta\le R_{2}$.
Let $g$ be a holomorphic function in the ellipse
\[
E=\{\k\in\C:\frac{(\Re\k)^{2}}{R_{1}^{2}}+\frac{(\Im\k)^{2}}{R_{2}^{2}}<1\}
\]
such that $\left|g(0)\right|\ge C_{0}>0$ and $\sup_{E}\left|g\right|\le D$.
There exists $\k\in(r,R_{1})$ such that
\[
\left|g(\k)\right|\ge C
\]
for some constant $C>0$ depending on $M$, $R_{1}-r$, $\eta$, $C_{0}$
and $D$ only.\end{lem}
\begin{proof}
Without loss of generality,
we can always suppose $R_{2}\le R_{1}$.

Set $\beta:=\sqrt{R_{1}^{2}+R_{2}^{2}}\le\sqrt{2}M$, $r_{i}=R_{i}/\beta$
and $\tilde{E}:=\{\k\in\C:\frac{(\Re\k)^{2}}{r_{1}^{2}}+\frac{(\Im\k)^{2}}{r_{2}^{2}}<1\}$.
The map $\psi_{1}\colon\tilde{E}\to E$, $\k\mapsto\beta\k$ is bi-holomorphic
and the segment $(r,R_{1})\subseteq E$ is transformed via $\psi_{1}^{-1}$
into $(r/\beta,R_{1}/\beta)\subseteq\tilde{E}$. Consider now a bi-holomorphic
transformation $\psi_{2}\colon B(0,1)\to\tilde{E}$. The existence
of this map is a consequence of the Riemann mapping theorem, and an
explicit formula is given in \cite[page 296]{nehari}. In particular,
\textcolor{black}{$\psi_{2}$} can be chosen so that \textcolor{black}{$\psi_{2}(0)=0$} and \textcolor{black}{$\psi_{2}^{-1}((r/\beta,R_{1}/\beta))=(r',1)$}
for some $r'\in(0,1)$. Since \textcolor{black}{$(R_{1}-r)/\beta\ge (R_{1}-r)/(\sqrt{2}M) $} and $1\le r_{1}/r_{2}=R_{1}/R_{2}\le M/\eta$
we have $1-r'\ge c$ \textcolor{black}{for some $c>0$ depending only on $M$, $R_{1}-r$, $\eta$,
$C_{0}$ and $D$}, as the ratio $r_{1}/r_{2}$ determines the deformation
carried out by $\psi_{2}$. Hence $r'\le\theta$ with $\theta=1-c$. 

Consider now the map $g'\colon B(0,1)\to\C$ defined by $g'=g\circ\psi_{1}\circ\psi_{2}$.
We have that $g'$ is holomorphic in $B(0,1)$, $\left|g'(0)\right|=\left|g(0)\right|\ge C_{0}$
and $\sup_{B(0,1)}\left|g'\right|=\sup_{E}\left|g\right|\le D$. By
Lemma~\ref{lem:momm} applied to $g'$ and $r'$ we obtain the result.
\end{proof}

\subsection{\label{sub:Proof-of-Theorem-1}The Helmholtz equation}

We prove here Theorem~\ref{thm:quantitative bounds-helmholtz-zeta}.
For simplicity, we shall say that a positive constant depends on a
priori data if it depends on $\Omega$,  $\maxa$, $\left|\A\right|$,
$M$, $C_{\zeta}$, $\left\Vert a\right\Vert _{\Cl^{\order-1,\alpha}(\overline{\Omega};\R^{d\times d})}$,
$\left\Vert (\epsilon,\sigma)\right\Vert _{W^{\order-1,\infty}(\Omega;\R)^{2}}$,
$\left\Vert \phi_{i}\right\Vert _{\Cl^{\order,\alpha}(\overline{\Omega};\C)}$
and $C_{0}$ only. \textcolor{black}{Recall that $D$ is given by \eqref{eq:definition of D} and that $\Sigma=\{\lambda_{l}:l\in\N^{*}\}$ denotes the set of the Dirichlet
eigenvalues of problem \eqref{eq:helmholtz-multi}. During the proof, we shall often refer to the results given in the Appendix.}

We first show that the map $\k\in D\mapsto u_{\k}^{i}\in C^{\order}$
is holomorphic. This will be one of the basic tools of the proof of
Theorem~\ref{thm:quantitative bounds-helmholtz-zeta}.
\begin{prop}
\label{pro: Helmholtz holomorphic}Under the assumptions of Theorem~\ref{thm:quantitative bounds-helmholtz-zeta},
the map
\[
D\longrightarrow\Cl^{\order}(\overline{\Omega};\C),\quad\k\longmapsto u_{\k}^{i}
\]
is holomorphic.\end{prop}
\begin{proof}
In view of Propositions~\ref{prop:helmoltz-wellposedness} and \ref{pro: Helmholtz regularity},
problem \eqref{eq:helmholtz-multi} is well-posed and $u_{\k}^{i}\in C^{\order}$.
If \eqref{eq:sigma equal 0-multi} holds, this result has already
been proved in \cite{alberti2013multiple}. The case where \eqref{eq:sigma bounds below-multi}
holds can be handled similarly \cite{albertidphil}.
\end{proof}

\textcolor{black}{Define for every $j=1,\dots,r$
\[
\theta^{j}\colon D\to\Cl(\overline{\Omega};\C),\;\k\mapsto\zeta^{j}\bigl(u_{\k}^{1},\dots,u_{\k}^{b}\bigr).
\]
As a consequence of the previous result, the maps $\theta^j$ are holomorphic.}
\begin{lem}
\label{lem:zeta j holomorphic}Under the hypotheses of Theorem~\ref{thm:quantitative bounds-helmholtz-zeta},
the map $\theta^{j}\colon D\to\Cl(\overline{\Omega};\C)$ is holomorphic
for all $j$.\end{lem}
\begin{proof}
It follows from Proposition~\ref{pro: Helmholtz holomorphic},
\eqref{eq:assumptions map zed-def} and Lemma~\ref{lem:analytic functions},
parts 2 and 3.
\end{proof}
\begin{figure}[t]
\caption{\label{fig:domain D}The domain $D$ and the admissible set \textcolor{red}{$\A$}.}

\begin{centering}
\subfloat[\label{fig:domain D-real}$D=\C\setminus\sqrt{\Sigma}$ if $\eqref{eq:sigma equal 0-multi}$
holds.]
{\includegraphics[scale=1]{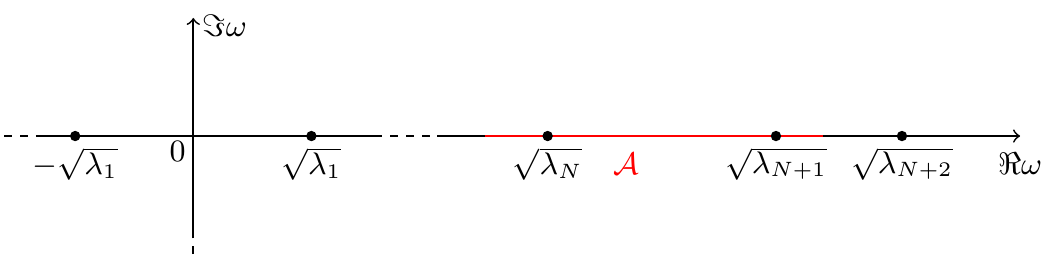}}
\par\end{centering}

\centering{}\subfloat[\label{fig:domain D-complex}$D=\{\k\in\C:\left|\Im\k\right|<\eta\}$
if $\eqref{eq:sigma bounds below-multi}$ holds.]{\includegraphics[scale=1]{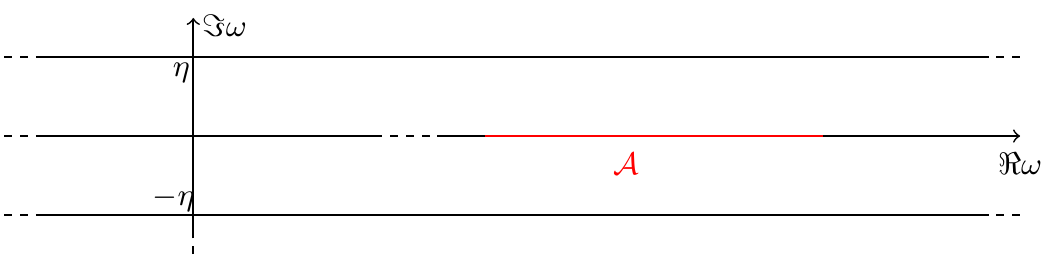}}
\end{figure}
We next study some a priori bounds on $\theta^{j}$ and $\partial_{\k}\theta^{j}$
(notation of Proposition~\ref{prop:helmoltz-wellposedness}).
\begin{lem}
\label{lem:bounds on u}Assume that the hypotheses of Theorem~\ref{thm:quantitative bounds-helmholtz-zeta}
hold true and take $j=1,\dots,r$ and $\k\in B(0,M)\cap D$.
\begin{enumerate}
\item If \eqref{eq:sigma equal 0-multi} holds true then there exists $C>0$
depending on a priori data such that

\begin{enumerate}
\item $\bigl\Vert\theta_{\k}^{j}\bigr\Vert_{C(\overline{\Omega};\C)}\le C\left[1+\sup_{l\in\N^{*}}\frac{1}{\left|\lambda_{l}-\k^{2}\right|}\right]^{s}$;
\item $\bigl\Vert\partial_{\k}\theta_{\k}^{j}\bigr\Vert_{C(\overline{\Omega};\C)}\le C\left[1+\sup_{l\in\N^{*}}\frac{1}{\left|\lambda_{l}-\k^{2}\right|}\right]^{s+2}$.
\end{enumerate}
\item If \eqref{eq:sigma bounds below-multi} holds true then there exists
$C>0$ depending on a priori data such that

\begin{enumerate}
\item $\bigl\Vert\theta_{\k}^{j}\bigr\Vert_{C(\overline{\Omega};\C)}\le C$;
\item $\bigl\Vert\partial_{\k}\theta_{\k}^{j}\bigr\Vert_{C(\overline{\Omega};\C)}\le C$.
\end{enumerate}
\end{enumerate}
\end{lem}
\begin{proof}
\textcolor{black}{We first prove part 1, namely we take $\sigma=0$.} In view of Proposition~\ref{prop:helmoltz-wellposedness}, part 1
and Proposition~\ref{pro: Helmholtz regularity} we have
\begin{equation}
\left\Vert u_{\k}^{i}\right\Vert _{C^{\order}(\overline{\Omega};\C)}\le C\left[1+\sup_{l\in\N^{*}}\frac{1}{\left|\lambda_{l}-\k^{2}\right|}\right],\qquad\k\in B(0,M)\cap D,\label{eq:bounds on u}
\end{equation}
whence we obtain part 1a from \eqref{eq:assumptions map zed-a}.

It can be easily seen that $\partial_{\k}u_{\k}^{i}$ is the solution
to
\[
\left\{ \begin{array}{l}
-\div(a\,\nabla(\partial_{\k}u_{\k}^{i}))-\k^{2}\epsilon\,\partial_{\k}u_{\k}^{i}=2\k\epsilon u_{\k}^{i}\qquad\text{in \ensuremath{\Omega},}\\
\partial_{\k}u_{\k}^{i}=0\qquad\text{on \ensuremath{\partial\Omega}.}
\end{array}\right.
\]
Arguing as before,  from Proposition~\ref{prop:helmoltz-wellposedness},
part 1 and Proposition~\ref{pro: Helmholtz regularity} we obtain
\begin{equation}
\left\Vert \partial_{\k}u_{\k}^{i}\right\Vert _{C^{\order}(\overline{\Omega};\C)}\le C\left[1+\sup_{l\in\N^{*}}\frac{1}{\left|\lambda_{l}-\k^{2}\right|}\right]^{2}.\label{eq:bounds on u'}
\end{equation}
Since $\partial_{\k}\theta_{\k}^{j}=D\zeta_{(u_{\k}^{1},\dots,u_{\k}^{b})}^{j}(\partial_{\k}u_{\k}^{1},\dots,\partial_{\k}u_{\k}^{b})$
we have
\[
\begin{split}\bigl\Vert\partial_{\k}\theta_{\k}^{j}\bigr\Vert_{C(\overline{\Omega};\C)} & =\bigl\Vert D\zeta_{(u_{\k}^{1},\dots,u_{\k}^{b})}^{j}(\partial_{\k}u_{\k}^{1},\dots,\partial_{\k}u_{\k}^{b})\bigr\Vert_{C(\overline{\Omega};\C)}\\
 & \le\bigl\Vert D\zeta_{(u_{\k}^{1},\dots,u_{\k}^{b})}^{j}\bigr\Vert_{\mathcal{B}(\Cl^{\order}(\overline{\Omega};\C){}^{b},\Cl(\overline{\Omega};\C))}\bigl\Vert(\partial_{\k}u_{\k}^{1},\dots,\partial_{\k}u_{\k}^{b})\bigr\Vert_{\Cl^{\order}(\overline{\Omega};\C){}^{b}}\\
 & \le C\left[1+\sup_{l\in\N^{*}}\frac{1}{\left|\lambda_{l}-\k^{2}\right|}\right]^{s+2},
\end{split}
\]
where the last inequality follows from \eqref{eq:assumptions map zed-b},
\eqref{eq:bounds on u} and \eqref{eq:bounds on u'}. Part 1b is now
proved.

Part 2 can be proved analogously, by using part 2 of Proposition~\ref{prop:helmoltz-wellposedness}
in lieu of part 1. The details are left to the reader.
\end{proof}
In the following two lemmata we study the case where \eqref{eq:sigma equal 0-multi}
holds true, and how to deal with the presence of the eigenvalues (see
Figure~\ref{fig:D-real-Atilde}).

\begin{figure}
\centering{}\caption{\label{fig:D-real-Atilde}The admissible sets \textcolor{red}{$\A$}
and \textcolor{blue}{$\tilde{\A}$}.}
\includegraphics[scale=1]{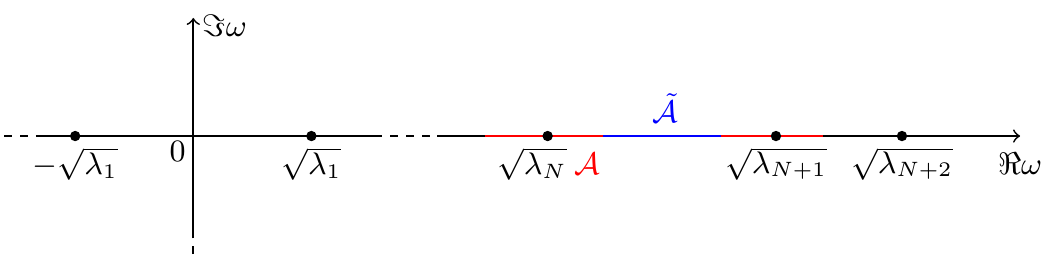}
\end{figure}

\begin{lem}
\label{lem:A'}Under the hypotheses of Theorem~\ref{thm:quantitative bounds-helmholtz-zeta},
assume that \eqref{eq:sigma equal 0-multi} holds true. Then there
exist $N\in\N^{*}$, $\delta>0$ and $\beta>0$ depending on $\Omega$,
$\maxa$, $\left|\A\right|$ and $M$ only and a closed interval $\tilde{\A}=[\tilde{K}_{min},\tilde{K}_{max}]\subseteq\A$
such that
\[
d(\tilde{\A}^{2},\Sigma)\ge\delta,\qquad\tilde{\A}^{2}\subseteq(\lambda_{l},\lambda_{l+1}),\qquad\bigl|\tilde{\A}\bigr|\ge\beta
\]
for some $l\le N$\textcolor{black}{, where $\tilde{\A}^{2}=\{\k^2:\k\in\tilde{\A}\}$.}
\end{lem}
\begin{proof}
In view of Lemma~\ref{lem:weyl's lemma} there exists $N\in\N^{*}$
depending on $\Omega$, $\maxa$ and $M$ only such that $[0,K_{max}^{2}]\cap\Sigma\subseteq\{\lambda_{1},\dots,\lambda_{N}\}$.
In particular, $\#(\A^{2}\cap\Sigma)\le N$. Therefore there exists
$l\le N$ such that $\left|\A^{2}\cap(\lambda_{l},\lambda_{l+1})\right|\ge\left|\A^{2}\right|(N+1)^{-1}$.
Write $\A^{2}\cap(\lambda_{l},\lambda_{l+1})=[p,q]$ and define $\tilde{\A}$
by $\tilde{\A}^{2}=[p+\frac{\left|\A^{2}\right|}{3(N+1)},q-\frac{\left|\A^{2}\right|}{3(N+1)}]$.
This concludes the proof, since $ $$\left|\A^{2}\right|$ depends
on $\left|\A\right|$ and $N$ only.
\end{proof}
Thanks to Lemma~\ref{lem:A'}, by taking a subinterval of the original
admissible set $\A$, without loss of generality we can assume that
\begin{equation}
d(\A^{2},\Sigma)\ge\delta,\qquad\A^{2}\subseteq(\lambda_{l},\lambda_{l+1}),\qquad l\le N\label{eq:new assumptions on A}
\end{equation}
for some $\delta>0$ and $N\in\N^{*}$ depending on $\Omega$, $\maxa$,
$\left|\A\right|$ and $M$ only. Moreover, the new size of $\A$
is comparable with the size of the original $\A$ by means of constants
depending on $\Omega$, $\maxa$, $\left|\A\right|$ and $M$ only.

The main idea is to apply Lemma~\ref{lem:momm-ellipse} to the maps
$\k\mapsto\theta_{\k}^{j}(x)$ and use the fact that in $\k=0$ they
are non-zero. However, in the case where \eqref{eq:sigma equal 0-multi}
holds true we first need to remove the singularities in the poles
$\pm\sqrt{\lambda_{1}},\dots,\pm\sqrt{\lambda_{N}}$.
\begin{lem}
\label{lem:poles}Under the hypotheses of Theorem~\ref{thm:quantitative bounds-helmholtz-zeta},
if \eqref{eq:sigma equal 0-multi} and \eqref{eq:new assumptions on A}
hold true then for any $x\in\Omega$ the function
\begin{equation}
\k\in B(0,K_{max})\longmapsto g_{x}^{j}(\k):=\theta_{\k}^{j}(x)\prod_{l=1}^{N}\frac{(\lambda_{l}-\k^{2})^{s}}{\lambda_{l}^{s}},\label{eq:g_x}
\end{equation}
is holomorphic in $B(0,K_{max})$ and
\[
\sup_{B(0,K_{max})}\bigl|g_{x}^{j}\bigr|\le C
\]
for some $C>0$ depending on a priori data.\end{lem}
\begin{proof}
Different positive constants depending on a priori data will be denoted
by $C$. In view of Lemma~\ref{lem:zeta j holomorphic}, the map
$\k\in\C\setminus\sqrt{\Sigma}\mapsto\theta_{\k}^{j}(x)\in\C$ is
holomorphic and by Lemma~\ref{lem:bounds on u}, part 1a, it is meromorphic
in $B(0,K_{max})$. For $\k\in B(0,K_{max})\cap D$ we have
\[
\begin{split}\bigl|g_{x}^{j}(\k)\bigr| & \le\bigl|\theta_{\k}^{j}(x)\bigr|\prod_{l=1}^{N}\frac{\left|\lambda_{l}-\k^{2}\right|^{s}}{\lambda_{l}^{s}}\\
 & \le C\lambda_{1}^{-Ns}\prod_{l=1}^{N}\left|\lambda_{l}-\k^{2}\right|^{s}\left[1+\sup_{l\in\N^{*}}\frac{1}{\left|\lambda_{l}-\k^{2}\right|^{s}}\right]\\
 & \le C\prod_{l=1}^{N}\left|\lambda_{l}-\k^{2}\right|^{s}\left[1+\sup_{l\le N}\frac{1}{\left|\lambda_{l}-\k^{2}\right|^{s}}+\sup_{l>N}\frac{1}{\left|\lambda_{l}-\k^{2}\right|^{s}}\right],
\end{split}
\]
where the second inequality follows from Lemma~\ref{lem:bounds on u},
part 1a. As a consequence 
\[
\begin{split}\bigl|g_{x}^{j}(\k)\bigr| & \le C\prod_{l=1}^{N}\left|\lambda_{l}-\k^{2}\right|^{s}\left[1+\sup_{l\le N}\frac{1}{\left|\lambda_{l}-\k^{2}\right|^{s}}\right]\\
 & \le C\left[\prod_{l=1}^{N}\left|\lambda_{l}-\k^{2}\right|^{s}+\frac{\prod_{l=1}^{N}\left|\lambda_{l}-\k^{2}\right|^{\textcolor{black}{s}}}{\inf_{l\le N}\left|\lambda_{l}-\k^{2}\right|^{\textcolor{black}{s}}}\right]\\
 & \le C,
\end{split}
\]
where the first inequality follows from
\[
\left|\lambda_{l}-\k^{2}\right|\ge\delta,\qquad l>N,
\]
and the third inequality from
\[
\left|\lambda_{l}-\k^{2}\right|\le2M^{2},\qquad l\le N.
\]
Therefore the map $g_{x}^{j}$ is holomorphic in $B(0,K_{max})$ and
$\sup_{B(0,K_{max})}\bigl|g_{x}^{j}\bigr|\le C$.
\end{proof}
The next lemma is the last step needed for the proof of Theorem~\ref{thm:quantitative bounds-helmholtz-zeta}. 
\begin{lem}
\label{lem:almost there}Under the hypotheses of Theorem~\ref{thm:quantitative bounds-helmholtz-zeta},
assume that if \eqref{eq:sigma equal 0-multi} holds then \eqref{eq:new assumptions on A}
holds. Then for every $x\in\Omega'$ there exists $\k_{x}\in\A$ such
that
\[
\bigl|\theta_{\k_{x}}^{j}(x)\bigr|\ge C,\qquad j=1,\dots,r
\]
for some $C>0$ depending on a priori data.\end{lem}
\begin{proof}
Several positive constants depending on a priori data will be denoted
by $C$.

\emph{First case -- Assumption \eqref{eq:sigma equal 0-multi}}. Take
$x\in\Omega'$ and define $g_{x}^{j}$ as in \eqref{eq:g_x}, where
$N$ is given by \eqref{eq:new assumptions on A}. Set
\[
g_{x}=\prod_{j=1}^{r}g_{x}^{j}.
\]
By Lemma~\ref{lem:poles} the map $g_{x}$ is holomorphic in $B(0,K_{max})$
and $\max_{B(0,K_{max})}\bigl|g_{x}\bigr|\le C$. Moreover, $\left|g_{x}(0)\right|\ge C_{0}^{r}$
by \eqref{eq:assumption k=00003D0}. Therefore, by Lemma~\ref{lem:momm-ellipse}
with $r=K_{min}$ and $R_{1}=R_{2}=K_{max}$ there exists $\k_{x}\in[r,R]=\A$
such that $\bigl|g_{x}(\k_{x})\bigr|\ge C$. As a consequence, in
view of \eqref{eq:g_x} we obtain
\[
\bigl|\prod_{j=1}^{r}\theta_{\k_{x}}^{j}(x)\bigr|=\bigl|g_{x}(\k_{x})\bigr|\prod_{l=1}^{N}\frac{\lambda_{l}^{rs}}{\left|\lambda_{l}-\k_{x}^{2}\right|^{rs}}\ge C,
\]
since $\lambda_{l}\ge\lambda_{1}\ge C(\Omega,\maxa)$ and $\left|\lambda_{l}-\k_{x}^{2}\right|\le2M^{2}$.
The result now follows from Lemma~\ref{lem:bounds on u}, part 1a.

\emph{Second case -- Assumption \eqref{eq:sigma bounds below-multi}}.
Take $x\in\Omega'$ and define 
\[
g_{x}(\k)=\prod_{j=1}^{r}\theta_{\k}^{j}(x),\quad\k\in D.
\]
In view of Lemma~\ref{lem:zeta j holomorphic}, the map $g_{x}$
is holomorphic in $D$ and by Lemma~\ref{lem:bounds on u}, part
2a, $\max_{B(0,M)\cap D}\bigl|g_{x}\bigr|\le C$. Moreover, $\left|g_{x}(0)\right|\ge C_{0}^{r}$
by \eqref{eq:assumption k=00003D0}. Therefore, by Lemma~\ref{lem:momm-ellipse}
with $r=K_{min}$, $R_{1}=K_{max}$, and $R_{2}=\eta$ there exists
$\k_{x}\in\A$ such that $\bigl|g_{x}(\k_{x})\bigr|\ge C$. The result
now follows from Lemma~\ref{lem:bounds on u}, part 2a.
\end{proof}
We are now ready to prove Theorem~\ref{thm:quantitative bounds-helmholtz-zeta}. 
\begin{proof}[Proof of Theorem~\ref{thm:quantitative bounds-helmholtz-zeta}]
Different positive constants depending on a priori data will be denoted
by $C$ or $Z$.

If \eqref{eq:sigma equal 0-multi} holds true, by Lemma~\ref{lem:A'}
we can assume \eqref{eq:new assumptions on A}. Thus, in view of Lemma~\ref{lem:almost there},
for every $x\in\Omega'$ there exists $\k_{x}\in\A$ such that
\[
\bigl|\theta_{\k_{x}}^{j}(x)\bigr|\ge C,\qquad j=1,\dots,r.
\]
Thus, by Lemma~\ref{lem:bounds on u}, parts 1b and 2b, there exists
$Z>0$ such that 
\begin{equation}
\bigl|\theta_{\k}^{j}(x)\bigr|\ge C,\qquad\k\in[\k_{x}-Z,\k_{x}+Z]\cap\A,\; j=1,\dots,r.\label{eq:2nd bound}
\end{equation}
Recall that $\A=[K_{min},K_{max}]$ and that $\k_{i}^{(n)}=K_{min}+\frac{(i-1)}{(n-1)}(K_{max}-K_{min})$.
It is trivial to see that there exists $P=P(Z,\left|\A\right|)\in\N$
such that
\begin{equation}
\A\subseteq\bigcup_{p=1}^{P}I_{p},\qquad I_{p}=[K_{min}+(p-1)Z,K_{min}+pZ].\label{eq:partition}
\end{equation}
Choose now $n\in\N$ big enough so that for every $p=1,\dots,P$ there
exists $i_{p}=1,\dots,n$ such that $\k(p):=\k_{i_{p}}^{(n)}\in I_{p}$.
Note that $n$ depends on $Z$ and $\left|\A\right|$ only.

Take now $x\in\Omega'$. Since $\left|[\k_{x}-Z,\k_{x}+Z]\right|=2Z$
and $\left|I_{p}\right|=Z$, in view of \eqref{eq:partition} there
exists $p_{x}=1,\dots,P$ such that $I_{p_{x}}\subseteq[\k_{x}-Z,\k_{x}+Z]$.
Therefore $\k(p_{x})\in[\k_{x}-Z,\k_{x}+Z]\cap\A$, whence by \eqref{eq:2nd bound}
there holds $\bigl|\theta_{\k(p_{x})}^{j}(x)\bigr|\ge C$ for all
$j=1,\dots,r$. Recalling the definition of $\theta^{j}$ this implies
\begin{equation}
\bigl|\zeta^{j}\bigl(u_{\k(p_{x})}^{1},\dots,u_{\k(p_{x})}^{b}\bigr)(x)\bigr|\ge C,\qquad j=1,\dots,r.\label{eq:last}
\end{equation}
Define now $\Omega'_{\k}=\{x\in\Omega':\min_{j}|\zeta^{j}\bigl(u_{\k}^{1},\dots,u_{\k}^{b}\bigr)(x)|>C/2\}$.
By \eqref{eq:last} this gives an open cover $\Omega'=\cup_{\k\in K^{(n)}}\Omega'_{\k}$,
since $\k(p_{x})\in K^{(n)}$. As a consequence, $K^{(n)}\times\{\phi_{1},\dots,\phi_{b}\}$
is $(\zeta,C/2)$-complete in $\Omega'$ (Definition~\ref{def:zeta-complete}).
The theorem is proved.
\end{proof}

\subsection{\label{sub:Proof-of-Corollary}\texorpdfstring{$(\zeta_{\det},C)$}{(z,C)}-complete sets
of measurements}

We now show how to apply Theorem~\ref{thm:quantitative bounds-helmholtz-zeta}
to the particular case of $(\zeta_{\det},C)$-complete sets. 
\begin{proof}[Proof of Corollary~\ref{cor:det-complete}]
The main point of the proof of this theorem is satisfying \eqref{eq:assumption k=00003D0} for \textcolor{black}{$\zeta=\zeta_{\det}$}.
Then, the result will follow immediately from Theorem~\ref{thm:quantitative bounds-helmholtz-zeta}. 

\emph{Case $d=2$.} It is sufficient to prove that
\[
\bigl|\zeta_{\det}^{j}\bigl(u_{0}^{1},u_{0}^{2},u_{0}^{3}\bigr)(x)\bigr|\ge\p_{0},\quad j=1,\dots,3,\; x\in\Omega'
\]
for some $C_{0}>0$ depending on $\Omega$, $\Omega'$, $\maxa$,
$\alpha$ and $\left\Vert a\right\Vert _{\Cl^{0,\alpha}(\overline{\Omega};\R^{2\times2})}$.

Several positive constants depending on $\Omega$, $\Omega'$, $\maxa$,
$\alpha$ and $\left\Vert a\right\Vert _{\Cl^{0,1}(\overline{\Omega};\R^{2\times2})}$
will be denoted by $C$. Recall that, setting ``$x_{0}=1$'', we
have
\[
\left\{ \begin{array}{c}
-\div(a\nabla u_{0}^{i})=0\qquad\text{in }\bo,\\
u_{0}^{i}=x_{i-1}\qquad\text{on }\bo.
\end{array}\right.
\]
Since $u_{0}^{1}=1$, the thesis is equivalent to show that
\begin{equation}
\left|\gamma(x)\right|:=\left|\det\begin{bmatrix}\nabla u_{0}^{2} & \nabla u_{0}^{3}\end{bmatrix}(x)\right|\ge C,\quad x\in\Omega'.\label{eq:thesis zeta^2}
\end{equation}
Fix now $x\in\Omega'$. Since $\Omega$ is convex, in view of Proposition~\ref{prop:aless-magn}
we have $\beta:=\left|\nabla u_{0}^{2}(x)\right|\ge C$. Set $\nabla^{\perp}u_{0}^{2}=(-\partial_{2}u_{0}^{2},\partial_{1}u_{0}^{2})$.
Therefore $\{\beta^{-1}\nabla u_{0}^{2}(x),\beta^{-1}\nabla^{\perp}u_{0}^{2}(x)\}$
is an orthonormal basis of $\R^{2}$. As a consequence there holds
\[
\nabla u_{0}^{3}(x)=(\nabla u_{0}^{3}(x)\cdot\beta^{-2}\nabla u_{0}^{2}(x))\nabla u_{0}^{2}(x)+(\nabla u_{0}^{3}(x)\cdot\beta^{-2}\nabla^{\perp}u_{0}^{2}(x))\nabla^{\perp}u_{0}^{2}(x).
\]
Setting $\xi=\nabla u_{0}^{3}(x)\cdot\beta^{-2}\nabla u_{0}^{2}(x)$
and $v=u_{0}^{3}-\xi u_{0}^{2}$, since $\gamma(x)=\nabla u_{0}^{3}(x)\cdot\nabla^{\perp}u_{0}^{2}(x)$
we have $\beta^{-2}\gamma(x)\nabla^{\perp}u_{0}^{2}(x)=\nabla v(x)$,
whence
\begin{equation}
\left|\gamma(x)\right|=\beta\left|\nabla v(x)\right|.\label{eq:gamma}
\end{equation}
Since $\Omega$ is convex and $v$ is the solution to
\[
\left\{ \begin{array}{c}
-\div(a\nabla v)=0\qquad\text{in }\bo,\\
v=x_{2}-\xi x_{1}\qquad\text{on }\bo,
\end{array}\right.
\]
we can apply again Proposition~\ref{prop:aless-magn} and obtain
$\left|\nabla v(x)\right|\ge C$ (note that $\left|\xi\right|\le C$
by standard elliptic regularity theory -- see Proposition~\ref{pro: Helmholtz regularity}).
As a consequence, in view of \eqref{eq:gamma} we obtain \eqref{eq:thesis zeta^2}.

\emph{Case $d=3$. }For simplicity, suppose first that $a=\hat{a}$.
Thus $u_{0}^{i}=x_{i-1}$ for $i=1,\dots,4$ (``$x_{0}=1$''). Therefore
\eqref{eq:assumption k=00003D0} is immediately satisfied with $C_{0}=1$.
The general case where $\left\Vert a-\hat{a}\right\Vert _{\Cl^{0,\alpha}}\le\delta$
can be handled by using a standard continuity argument. More precisely,
we obtain $\left\Vert u_{0}^{i}-x_{i-1}\right\Vert _{\Cl^{1}}\le c\delta$,
and so \eqref{eq:assumption k=00003D0} is satisfied provided that
$\delta$ is chosen small enough (for details, see \cite{alberti2013multiple}).
\end{proof}

\subsection{\label{sub:Proof-of-Theorem-2}Maxwell's equations}

As in the case of the Helmholtz equation, the basic tool to prove
Theorem~\ref{thm:quantitative bounds-complex maxwell-zeta} is the
holomorphicity of the map $\k\mapsto(E_{\k}^{i},H_{\k}^{i})\in C^{\order}$.
\begin{prop}[\cite{albertigsII}]
\label{pro: maxwell holomorphic}Under the assumptions of Theorem~\ref{thm:quantitative bounds-complex maxwell-zeta},
the map
\[
\{\k\in B(0,M):|\Im\k|<\eta\}\longrightarrow\Cl^{\order}(\overline{\Omega};\C^{6}),\quad\k\longmapsto(E_{\k}^{i},H_{\k}^{i})
\]
is holomorphic, where $\eta>0$ is given by Proposition~\ref{prop:well-posedness-maxwell}.
\end{prop}
The rest of the proof of Theorem~\ref{thm:quantitative bounds-complex maxwell-zeta}
is very similar to the proof of Theorem~\ref{thm:quantitative bounds-helmholtz-zeta}
in the case where \eqref{eq:sigma bounds below-multi} holds true.
The results of $\S$~\ref{sub:Maxwell's-equations-appendix} have
to be used in place of the corresponding results of $\S$~\ref{sub:The-Helmholtz-equation-appendix}.
If $\sigma=\hat{\sigma}$, no further investigation is needed. If
$\left\Vert \sigma-\so\right\Vert _{W^{\order+1,p}}\le\delta$, an
argument similar to the one given in the proof of Corollary~\ref{cor:det-complete}
in the 3D case can be used \cite{albertigsII}. The details are omitted.

\section{\label{sec:Applications-to-hybrid}Applications to hybrid imaging
inverse problems}

In this section we apply the theory presented so far to three examples
of hybrid imaging problems. The reader is referred to \cite{albertigsII,albertidphil,ammari2014admittivity}
for other relevant examples.

\subsection{\label{sub:Applications-to-Microwave}Microwave imaging by ultrasound
deformation}

We consider the hybrid problem arising from the combination of microwaves
and ultrasounds that was introduced in \cite{cap2011}. The problem
is modelled by the Helmholtz equation \eqref{eq:helmholtz-multi}.
In addition to the previous assumptions, we suppose that $a$ is scalar-valued
and $ $$\sigma=0$. In microwave imaging, $a$ is the inverse of
the magnetic permeability, $\epsilon$ is the electric permittivity
and $\A=[K_{\text{min}},K_{\text{max}}]$ represent the admissible
frequencies in the microwave regime.

Given a set of measurements $K\times\left\{ \phi_{i}\right\} $ we
consider internal data of the form
\[
e_{\k}^{ij}=\epsi\, u_{\k}^{i}u_{\k}^{j},\qquad E_{\k}^{ij}=a\,\nabla u_{\k}^{i}\cdot\nabla u_{\k}^{j}.
\]
For simplicity, we denote $e_{\k}=(e_{\k}^{ij})_{ij}$ and similarly
for $E$. These internal energies have to be considered as known functions
in some subdomain $\Omega'\Subset\Omega$.

We need to choose a suitable set $K\times\left\{ \phi_{i}\right\} $
and  find $a$ and $\epsilon$ in $\Omega'$ from the knowledge of
$e_{\k}^{ij}$ and $E_{\k}^{ij}$ in $\Omega'$. This can be achieved
via two reconstruction formulae for $a/\epsilon$ and $\epsilon$,
respectively. Their applicability is guaranteed if $K\times\left\{ \phi_{i}\right\} $
is $(\zeta_{\times},C)$-complete, where $\zeta_{\times}\colon\Cl^{1}(\overline{\Omega};\C){}^{3}\longrightarrow\Cl(\overline{\Omega};\C)^{2}$
is given by
\[
\zeta_{\times}(u^{1},u^{2},u^{3})=\begin{cases}
\bigl(u^{1},\nabla u^{2}\times\nabla u^{3}\bigr) & \text{if }d=2,\\
\bigl(u^{1},(\nabla u^{2}\times\nabla u^{3})_{3}\bigr) & \text{if }d=3.
\end{cases}
\]
Note that $\zeta_{\times}^{2}=\zeta_{\det}^{2}$ in two dimensions,
but if $d=3$ then only two linearly independent gradients are required
with $\zeta_{\times}^{2}$. Thus, $(\zeta_{\times},C)$-complete sets
can be constructed by arguing as in Corollary~\ref{cor:det-complete}.
In particular, under the assumptions of Corollary~\ref{cor:det-complete},
a suitable choice for the boundary conditions is $\phi_{1}=1$, $\phi_{2}=x_{1}$
and $\phi_{3}=x_{2}$. The reconstruction algorithm with the use of
multiple frequencies was detailed in \cite{alberti2013multiple}.
Only the main steps are presented here.

Let $K\times\left\{ \phi_{1},\phi_{2},\phi_{3}\right\} $ be a $(\zeta_{\times},C)$-complete
set of measurements in $\Omega'$. As in Definition~\ref{def:zeta-complete},
this gives an open cover $\Omega'=\cup_{\k\in K\cap D}\Omega'_{\k}$
such that
\[
\bigl|u_{\k}^{1}\bigr|\ge C,\quad\bigl|\nabla u_{\k}^{2}\times\nabla u_{\k}^{3}\bigr|\ge C\quad\text{in }\Omega'_{\k}.
\]
These constraints allow to apply the following reconstruction procedure.
\begin{prop}[\cite{alberti2013multiple}]
\label{pro:formulaaqandq} Suppose that for all $\k\in K\cap D$
and $i=1,2,3$, $\left\Vert e_{\k}^{ii}\right\Vert _{L^{\infty}(\Omega')}\le F$
and $\left\Vert E_{\k}^{ii}\right\Vert _{L^{\infty}(\Omega')}\le F$
for some $F>0$.
\begin{enumerate}
\item There exists $c>0$ depending on $\maxa$ and $F$ such that for any
$\k\in K\cap D$ 
\[
\left|\nabla(e_{\k}/\tr(e_{\k}))\right|_{2}^{2}\ge cC^{6}\qquad\text{in }\Omega_{\k}'
\]
and $a/\epsilon$ is given in terms of the data by
\[
\frac{a}{\epsi}=2\,\frac{\tr(e_{\k})\,\tr(E_{\k})-\tr(e_{\k}E_{\k})}{\tr(e_{\k})^{2}\left|\nabla(e_{\k}/\tr(e_{\k}))\right|_{2}^{2}}\qquad\text{in }\Omega_{\k}'.
\]

\item Moreover, if $\epsilon\in\Honer$ then $\log\epsilon$ is the unique
solution to the problem
\[
\left\{ \begin{array}{l}
-\div\left(\frac{a}{\epsilon}\,\sum_{\k}e_{\k}^{11}\,\nabla u\right)=-\div\left(\frac{a}{\epsilon}\,\nabla\left(\sum_{\k}e_{\k}^{11}\right)\right)+2\sum_{\k}\left(E_{\k}^{11}-\k e_{\k}^{11}\right)\quad\text{in }\Omega',\\
u=\log\epsilon_{|\partial\Omega'}\qquad\text{on \ensuremath{\partial\Omega}}'.
\end{array}\right.
\]

\end{enumerate}
\end{prop}

\subsection{\label{sub:Thermo-acoustic-tomography}Quantitative thermo-acoustic
tomography (QTAT)}

In thermo-acoustic tomography \cite{patch-scherzer-2007}, the combination
of acoustic waves and microwaves is carried out in a different way,
if compared to the hybrid problem studied in $\S$~\ref{sub:Applications-to-Microwave}.
The absorption of the microwaves inside the object results in local
heating, and so in a local expansion of the medium. This creates acoustic
waves that propagate outside the domain, where they can be measured.
In a first step \cite{Mathematics-of-thermoacoustic-tomography,bal2011quantitative},
it is possible to measure the amount of absorbed radiation, which
is given by
\[
e_{\k}^{ii}(x)=\sigma(x)\left|u_{\k}^{i}(x)\right|^{2},\qquad x\in\Omega,
\]
where $\Omega\subseteq\R^{d}$ is a smooth bounded domain, $d=2,3$,
$u_{\k}^{i}$ is the unique solution to
\begin{equation}
\left\{ \begin{array}{l}
-\Delta u_{\k}^{i}-(\k^{2}+\ii\k\sigma)u_{\k}^{i}=0\qquad\text{in \ensuremath{\Omega},}\\
u_{\k}^{i}=\phi_{i}\qquad\text{on \ensuremath{\partial\Omega},}
\end{array}\right.\label{eq:physical_modeling-helmholtz-thermo}
\end{equation}
and $\sigma\in L^{\infty}(\Omega;\R)$ satisfies \eqref{eq:sigma bounds below-multi}.
The problem of QTAT is to reconstruct $\sigma$ from the knowledge
of $e_{\k}^{ij}$, where $e_{\k}^{ij}$ represent the polarised data
\[
e_{\k}^{ij}(x)=\sigma(x)u_{\k}^{i}(x)\overline{u_{\k}^{j}(x)},\qquad x\in\Omega.
\]

We shall see that it is possible to reconstruct $\sigma$ if $K\times\{\phi_{1},\dots,\phi_{d+1}\}$
is a $(\zeta'_{\det},C)$-complete set, where $\zeta'_{\det}\colon\Cl^{1}(\overline{\Omega};\C)^{d+1}\to\Cl(\overline{\Omega};\C)^{2}$
is given by
\[
\zeta'_{\det}(u^{1},\dots,u^{d+1})=\Bigl(u^{1},\det\begin{bmatrix}u^{1} & \cdots & u^{d+1}\\
\nabla u^{1} & \cdots & \nabla u^{d+1}
\end{bmatrix}\Bigr).
\]
Since $a=1$, the construction of $(\zeta'_{\det},C)$-complete sets
of measurements can be easily achieved with the multi-frequency approach
in any dimensions.
\begin{prop}
\label{prop:zeta_primo_det}Assume that $a=\epsilon=1$ and that $\sigma\in L^{\infty}(\Omega;\R)$
satisfies \eqref{eq:sigma bounds below-multi}. Then there exist $C>0$
and $n\in\N$ depending on $\Omega$, $\maxa$, $M$ and $\left|\A\right|$
only such that
\[
K^{(n)}\times\{1,x_{1},\dots,x_{d}\}
\]
is a $(\zeta'_{\det},C)$-complete set of measurements in $\Omega$.\end{prop}
\begin{proof}
It follows immediately from Theorem~\ref{thm:quantitative bounds-helmholtz-zeta},
since the assumption $a=1$ yields \eqref{eq:assumption k=00003D0}
with $C_{0}=1$.
\end{proof}
Let $K\times\{\phi_{1},\dots,\phi_{d+1}\}$ be a $(\zeta'_{\det},C)$-complete
set in $\Omega$. As in the previous subsection, this gives an open
cover $\Omega=\cup_{\k\in K}\Omega{}_{\k}$ such that for any $\k\in K$
and $x\in\Omega_{\k}$
\begin{equation}
\left|u_{\k}^{1}\right|(x)\ge C,\qquad\bigl|\det\begin{bmatrix}u_{\k}^{1} & \cdots & u_{\k}^{d+1}\\
\nabla u_{\k}^{1} & \cdots & \nabla u_{\k}^{d+1}
\end{bmatrix}(x)\bigr|\ge C.\label{eq:qtat}
\end{equation}
With this assumption, it is possible to apply the following reconstruction
formula in each subdomain $\Omega_{\k}$. We use the notation $\alpha_{\k}^{i}=e_{\k}^{i1}/e_{\k}^{11}$
and $A_{\k}=\begin{bmatrix}\nabla\alpha_{\k}^{2} & \cdots & \nabla\alpha_{\k}^{d+1}\end{bmatrix}$:
these quantities are well defined if \eqref{eq:qtat} is satisfied.
\begin{prop}[{\cite[Theorem 3.3]{ammari2012quantitative}}]
\label{prop:formula_thermo-acoustic} Assume that \eqref{eq:qtat}
holds in a subdomain $\tilde{\Omega}\subseteq\Omega$. There exists
$c>0$ depending on $\Omega$, $\maxa$ and $M$ such that $\left|\det A_{\k}\right|\ge cC$
in $\tilde{\Omega}$, and $\sigma$ can be reconstructed via
\[
\sigma=\frac{-\Re v_{\k}\cdot\Im v_{\k}+\div\Im v_{\k}}{2\k}\quad\text{in }\tilde{\Omega},
\]
where $v_{\k}=A_{\k}^{-1}\div(A_{\k})^{T}$ (the divergence acts on
each column).
\end{prop}
In \cite{ammari2012quantitative}, in order to find suitable illuminations
to satisfy \eqref{eq:qtat}, complex geometric optics solutions are
used; these have several drawbacks, as it was discussed in Section~\ref{sec:Introduction}.
Proposition~\ref{prop:zeta_primo_det} gives a priori simple illuminations
and a finite number of frequencies to satisfy the desired constraints
in each $\Omega_{\k}$. By Proposition~\ref{prop:formula_thermo-acoustic},
$\sigma$ can be reconstructed everywhere thanks to the cover $\Omega=\cup_{\k\in K}\Omega{}_{\k}$.

\subsection{\label{sub:Reconstruction-of-first method-intro}Magnetic resonance
electrical impedance tomography (MREIT)}

In this example, we model the problem with the Maxwell's equations
\eqref{eq:combined i-maxwell}. Combining electric currents with an
MRI scanner, we can measure the internal magnetic fields $H_{\k}^{\phi_{i}}$
\cite{seo2011magnetic,seo-kim-etal-2012}. Assuming $\mu=1$, the
electromagnetic parameters to image are $\epsilon$ and $\sigma$,
and both are assumed isotropic. We present here only a sketch of the
use of the multi-frequency technique to this problem: full details
are given in \cite{albertigsII}.

We shall see that $(\zeta_{\det}^{M},C)$-complete sets are sufficient
to be able to image the electromagnetic parameters (Example~\ref{exa:zeta_det_maxwell}).
The construction of $(\zeta_{\det}^{M},C)$-complete sets is an immediate
consequence of Corollary~\ref{cor:zeta-complete}.
\begin{prop}
\label{prop:det-complete-maxwell}Assume that \eqref{eq:assumption_coefficients-maxwell-multi}
holds with $\order=0$ and let $\so\in\R^{3\times3}$ satisfy \eqref{eq:assumption_ell-maxwell-multi}.
There exist $\delta>0$ and $C>0$ depending on $\Omega$,  $\maxa$,
$\left|\A\right|$, $M$ and $\left\Vert (\mu,\epsi,\sigma)\right\Vert _{W^{1,p}(\overline{\Omega};\R^{3\times3})^{3}}$
such that if $\left\Vert \sigma-\so\right\Vert _{W^{1,p}(\overline{\Omega};\R^{3\times3})}\le\delta$
then 
\[
K^{(n)}\times\{\e_{1},\e_{2},\e_{3}\}
\]
is a $(\zeta_{\det}^{M},C)$-complete set of measurements.\end{prop}
\begin{proof}
We want to apply Corollary~\ref{cor:zeta-complete} with $\zeta=\zeta_{\det}^{M}$
and $\psi_{i}=x_{i}$ for $i=1,2,3$. We only need to show that \eqref{eq:assumption k 0-cor}
holds. Since $w^{i}=x_{i}$, for every $x\in\Omega$ there holds $\zeta\bigl(\nabla w^{1},\nabla w^{2},\nabla w^{3}\bigr)(x)=\det\begin{bmatrix}\e_{1} & \e_{2} & \e_{3}\end{bmatrix}=1$,
as desired.
\end{proof}
Let $K\times\{\phi_{1},\phi_{2},\phi_{3}\}$ be a $(\zeta_{\det}^{M},C)$-complete
set. With the notation of Definition~\ref{def:zeta-complete-maxwell},
there is an open cover $\Omega=\cup_{\k}\Omega_{\k}$ such that
\begin{equation}
\bigl|\det\begin{bmatrix}E_{\k}^{1} & E_{\k}^{2} & E_{\k}^{3}\end{bmatrix}\bigr|>0\qquad\text{in }\Omega_{\k}.\label{eq:mreit}
\end{equation}
A simple calculation shows that $q_{\k}=\k\epsi+\ii\sigma$ satisfies
a first order partial differential equation of the form
\[
\nabla q_{\k}M_{\k}=F(\k,q_{\k},H_{\k}^{i},\Delta H_{\k}^{i})\qquad\text{in \ensuremath{\Omega},}
\]
where $M_{\k}$ is the $3\times6$ matrix-valued function given by
\[
M_{\k}=\left[\begin{array}{ccccc}
\curl H_{\k}^{1}\times\e_{1} & \curl H_{\k}^{1}\times\e_{2} & \cdots & \curl H_{\k}^{3}\times\e_{1} & \curl H_{\k}^{3}\times\e_{2}\end{array}\right],
\]
and $F$ is a given vector-valued function. If $\bigl|\det\begin{bmatrix}E_{\k}^{1} & E_{\k}^{2} & E_{\k}^{3}\end{bmatrix}(x)\bigr|>0$,
then it is easy to see that $M_{\k}(x)$ admits a right inverse $M_{\k}^{-1}(x)$.
By \eqref{eq:mreit}, $M_{\k}$ is invertible in $\Omega_{\k}$. The
equation for $q_{\k}$ becomes
\begin{equation}
\nabla q_{\k}=F(\k,q_{\k},H_{\k}^{i},\Delta H_{\k}^{i})M_{\k}^{-1}\quad\text{in }\Omega_{\k}.\label{eq:final-1}
\end{equation}
Proceeding as in \cite{bal2012inversediffusion}, it is possible to
integrate \eqref{eq:final-1} in each $\Omega_{\k}$ and reconstruct
$q_{\k}$ uniquely, provided that $q_{\k}$ is known at one point
of $\Omega$ \cite{albertigsII}.

\section{\label{sec:Conclusions}Conclusions}

Motivated by several hybrid imaging inverse problems, we studied the
boundary control of solutions of the Helmholtz and Maxwell equations
to enforce local non-zero constraints inside the domain. We have improved
the multiple frequency approach to this problem introduced in \cite{alberti2013multiple,albertigsII}
and have shown its effectiveness in several contexts. More precisely,
we give a priori boundary conditions $\phi_{i}$ and a finite set
of frequencies $K^{(n)}$ such that the corresponding solutions $u_{\k}^{i}$
satisfy the required constraints with an a priori determined constant.

An open problem concerns a more precise estimation of the number of
needed frequencies $n$. It is possible to show that, under the assumption
of real analytic coefficients, almost any $d+1$ frequencies in a
fixed range give the required constraints, where $d$ is the dimension
of the space \textcolor{black}{\cite{capalb-analytic}}. The proof is based on the structure
of analytic varieties, and so the hypothesis of real analyticity is
crucial. However, this assumption is far too strong for the applications.
Thus, a natural question to ask is whether it is possible to lower
the assumption of real analyticity.

Satisfying the constraints in the case $\k=0$ is usually straightforward
in two dimensions, but may present difficulties in 3D if $a$ (or
$\sigma$ in the case of Maxwell's equations) is not constant. The
method may work even if the constraint is not verified in the case
$\k=0$: when dealing with the constraints $\left|\nabla u_{\k}\right|\ge C$,
a generic choice of the boundary condition $\phi$ is sufficient \cite{alberti-genericity}.
However, choosing a generic boundary condition may give a very low
constant $C$ and a very high number of frequencies. An open problem
is to find an alternative to the study of the constraints in $\k=0$.
In particular, as far as the Helmholtz equation is concerned, an asymptotic
expansion of $u_{\k}$ for high frequencies $\k$  may give the required
non-zero constraints, and by holomorphicity this would still give
the desired result.

\section*{Acknowledgments}

This work has been done during my D.Phil.~at the Oxford Centre for
Nonlinear PDE under the supervision of Yves Capdeboscq, whom I would
like to thank for many stimulating discussions on the topic and for
a thorough revision of the paper. It is a pleasure to thank Giovanni
Alessandrini for suggesting to me the ideas of Lemma~\ref{lem:A'}.
I was supported by the EPSRC Science \& Innovation Award to the Oxford
Centre for Nonlinear PDE (EP/EO35027/1).

\bibliographystyle{abbrv}
\bibliography{alberti}

\appendix

\section{\label{sec:Some-basic-tools}Some basic tools}

\subsection{\label{sub:The-Helmholtz-equation-appendix}The Helmholtz equation}

The following result concerns the well-posedness for the Helmholtz
equation. The result is standard: for a proof, see \cite{albertidphil}.
\begin{prop}
\label{prop:helmoltz-wellposedness}Assume that \eqref{eq:ellipticity_a-epsi-multi}
holds and take $M>0$.
\begin{enumerate}
\item [1.]If \eqref{eq:sigma equal 0-multi} holds then there exists $\Sigma=\left\{ \lambda_{l}:l\in\N^{*}\right\} \subseteq\R_{+}$
with \textup{$\lambda_{l}\to+\infty$} such that for $\k\in(\C\setminus\sqrt{\Sigma})\cap B(0,M)$,
$f\in\Vp$ and $\phi\in\Hone$ the problem
\begin{equation}
\left\{ \begin{array}{l}
-\div(a\,\nabla u)-\k^{2}\,\epsilon\, u=f\qquad\text{in \ensuremath{\Omega},}\\
u=\phi\qquad\text{\text{on \ensuremath{\partial\Omega},}}
\end{array}\right.\label{eq:Helboundary-real}
\end{equation}
has a unique solution $u\in\Hone$ and
\begin{equation}
\left\Vert u\right\Vert _{\Hone}\le C(\Omega,\Lambda,M)\Bigl[1+\sup_{l\in\N^{*}}\frac{1}{\left|\lambda_{l}-\k^{2}\right|}\Bigr]\bigl(\left\Vert \phi\right\Vert _{\Hone}+\left\Vert f\right\Vert _{\Vp}\bigr).\label{eq:helboundstab}
\end{equation}

\item [2.]If \eqref{eq:sigma bounds below-multi} holds then there exists
$\eta>0$ depending on $\Omega$ and $\maxa$ only such that for $\k\in B(0,M)$
with $\Im\k\ge-\eta$, $f\in\Vp$ and $\phi\in\Hone$ the problem
\begin{equation}
\left\{ \begin{array}{l}
-\div(a\,\nabla u)-(\k^{2}\epsilon+\ii\k\sigma)\, u=f\qquad\text{in \ensuremath{\Omega},}\\
u=\phi\qquad\text{\text{on \ensuremath{\partial\Omega},}}
\end{array}\right.\label{eq:Helboundary-complex}
\end{equation}
has a unique solution $u\in\Hone$ with
\begin{equation}
\left\Vert u\right\Vert _{\Hone}\le C(\Omega,\maxa,M)\left[\left\Vert \phi\right\Vert _{\Hone}+\left\Vert f\right\Vert _{\Vp}\right].\label{eq:helboundstab-1}
\end{equation}

\end{enumerate}
\end{prop}
We have the following result, regarding the asymptotic distribution
of the eigenvalues. The result is classical and  is known as Weyl's
lemma.
\begin{lem}
\label{lem:weyl's lemma}Assume that \eqref{eq:ellipticity_a-epsi-multi}
and \eqref{eq:sigma equal 0-multi} hold true. There exist $C_{1},C_{2}>0$
depending on $\Omega$ and $\maxa$ such that
\[
C_{1}l^{\frac{2}{d}}\le\lambda_{l}\le C_{2}l^{\frac{2}{d}},\qquad l\in\N^{*}.
\]
\end{lem}
\begin{proof}
Let $\mathfrak{F}_{l}$ denote the set of all $l$-dimensional subspaces
of $\Vr$. In view of the Courant\textendash{}Fischer\textendash{}Weyl
min-max principle \cite[Exercise 12.4.2]{schmudgen2012} we have
\[
\lambda_{l}=\min_{D\in\mathfrak{F}_{l}}\max_{u\in D\setminus\{0\}}\frac{\int_{\Omega}a\nabla u\cdot\nabla u\, dx}{\int\epsilon u^{2}\, dx},\qquad l\in\N^{*}.
\]
Therefore we have
\begin{equation}
\Lambda^{-2}\mu_{l}\le\lambda_{l}\le\Lambda^{2}\mu_{l},\qquad l\in\N^{*},\label{eq:lambda_mu}
\end{equation}
where $\mu_{l}=\min_{D\in\mathfrak{F}_{l}}\max_{u\in D\setminus\{0\}}(\int_{\Omega}\nabla u\cdot\nabla u\, dx)(\int u^{2}\, dx)^{-1}$.
By the min-max principle, $\mu_{l}$ are the eigenvalues of the Laplace
operator on $\Omega$, and so they satisfy 
\[
c_{1}l^{\frac{2}{d}}\le\mu_{l}\le c_{2}l^{\frac{2}{d}},\qquad l\in\N^{*}
\]
for some $c_{1},c_{2}>0$ depending on $\Omega$ (see \cite[Theorem 12.14]{schmudgen2012}
or \cite[Chapter 5, Lemma 3.1]{kavian1993}). Combining this inequality
with \eqref{eq:lambda_mu} yields the result.
\end{proof}
We now study regularity for the Helmholtz equation, which is a consequence
of classical elliptic regularity theory \cite[Theorem 5.21]{giaquinta_martinazzi}. \textcolor{black}{ For $\order\in\N$, we use the notation}
\[
 \textcolor{black}{X_\order=\begin{cases}
           L^{\infty}(\Omega;\C) & \text{if $\order=0,1$,}\\
           \Cl^{\order -2,\alpha}(\overline{\Omega};\C) & \text{if $\order\ge 2$.}
          \end{cases}}
\]
\begin{prop}
\label{pro: Helmholtz regularity}Take $\order\in\N$, $\alpha\in(0,1)$
and $M>0$. Assume that \eqref{eq:ellipticity_a-epsi-multi}, \eqref{eq:regularity assumption-helmh-multi}
and either \eqref{eq:sigma equal 0-multi} or \eqref{eq:sigma bounds below-multi}
hold. \textcolor{black}{ Take $\k\in\C$ with $\left|\k\right|\le M$, $f\in X_\order$, $F\in X_{\order+1}^3$ and $\phi\in\Cl^{\order,\alpha}(\overline{\Omega};\C)$.} Let $u\in\Hone$ be a solution to
\[
\left\{ \begin{array}{l}
-\div(a\,\nabla u)-(\k^{2}\epsilon+\ii\k\sigma)\, u=\div F+f\qquad\text{in \ensuremath{\Omega},}\\
u=\phi\qquad\text{\text{on \ensuremath{\partial\Omega}.}}
\end{array}\right.
\]
 Then $u\in\Cl^{\order,\alpha}(\overline{\Omega};\C)$ and
\[
\left\Vert u\right\Vert _{\Cl^{\order}(\overline{\Omega};\C)}\le C\left(\left\Vert u\right\Vert _{\Hone}+\left\Vert \phi\right\Vert _{\Cl^{\order,\alpha}(\overline{\Omega};\C)}+\left\Vert f\right\Vert _{X_\order}+\left\Vert F\right\Vert _{X_{\order+1}^3}\right)
\]
for some $C>0$ depending only on $\Omega$,  $\maxa$, $\order$,
$\alpha$, $M$, $\left\Vert (\epsilon,\sigma)\right\Vert _{W^{\order-1,\infty}(\Omega;\R)^{2}}$
and $\left\Vert a\right\Vert _{\Cl^{\order-1,\alpha}(\overline{\Omega};\R^{d\times d})}$.
\end{prop}

\subsection{\label{sub:Maxwell's-equations-appendix}Maxwell's equations}

We first study well-posedness for Maxwell's equations. The result
is standard: for a proof, see \cite{albertigsII,albertidphil}.
\begin{prop}
\label{prop:well-posedness-maxwell}Assume that \eqref{eq:assumption_coefficients-maxwell-multi}
holds and take $M>0$. There exist $\eta,C>0$ depending on $\Omega$,
 $\Lambda$ and $M$ such that for all $\k\in\C$ with $\left|\Im\k\right|\le\eta$
and $\left|\k\right|\le M$ the problem
\begin{equation}
\left\{ \begin{array}{l}
\curl E_{\k}=\ii\k\mu H_{\k}\qquad\text{in \ensuremath{\Omega},}\\
\curl H_{\k}=-\ii(\k\epsilon+\ii\sigma)E_{\k}\qquad\text{in \ensuremath{\Omega},}\\
E_{\k}\times\nu=\phi\times\nu\qquad\text{on \ensuremath{\partial\Omega},}
\end{array}\right.\label{eq:combined-maxwell}
\end{equation}
admits a unique solution $(E_{\k},H_{\k})$ in $\Hcurl\times\Hmu$
satisfying
\[
\left\Vert (E_{\k},H_{\k})\right\Vert _{\Hcurl^{2}}\le C\left\Vert \phi\right\Vert _{\Hcurl}.
\]

\end{prop}
Next, regularity properties are discussed. This result follows from
the regularity theory for Maxwell's equations described in \cite{alberti-capdeboscq}
and is proven in detail in \cite{albertigsII,albertidphil}.
\begin{prop}
\label{prop:regularity-maxwell}Assume that \eqref{eq:assumption_coefficients-maxwell-multi}
holds for some $p>3$ and $\order\in\N$. Take $\eta,M>0$ as in Proposition~\ref{prop:well-posedness-maxwell}.
For $\k\in\C$ with $\left|\Im\k\right|\le\eta$ and $\left|\k\right|\le M$
let $(\E,\H)$ be the unique solution in $\Hcurl\times\Hmu$ to \eqref{eq:combined-maxwell}.
Then $(\E,\H)\in\Cl^{\order}(\overline{\Omega};\C^{6})$ and
\[
\left\Vert (\E,\H)\right\Vert _{\Cl^{\order}(\overline{\Omega};\C^{6})}\le C\left\Vert \phi\right\Vert _{W^{\order+1,p}(\Omega;\C^{3})}
\]
for some $C>0$ depending on $\Omega$,  $\maxa$, $M$, $\kappa$,
$p$ and $\left\Vert (\mu,\epsilon,\sigma)\right\Vert _{W^{\order+1,p}(\Omega;\R^{3\times3})^{3}}$
only.
\end{prop}

\subsection{\label{sub:The-absence-of}The critical points of solutions to the
conductivity equation}

We start with a qualitative property for solutions to the conductivity
equation.
\begin{lem}[{\cite[Theorem 2.7]{alessandrinimagnanini1994}}]
\label{lem:ales-magn}Let $\Omega\subseteq\R^{2}$ be a smooth and
bounded domain and take $\Omega'\Subset\Omega$. Let $a\in\Cl^{0,\alpha}(\overline{\Omega};\R^{2\times2})$
be such that \eqref{eq:ellipticity_a-multi} holds true and $\phi\in\Cl^{1,\alpha}(\overline{\Omega};\R)$
be such that $\phi_{|\bo}$ has one minimum and one maximum. Then
the solution $u\in\Cl^{1}(\overline{\Omega};\R)$ to 
\[
\left\{ \begin{array}{l}
-\div(a\nabla u)=0\qquad\text{in \ensuremath{\Omega},}\\
u=\phi\qquad\text{on \ensuremath{\partial\Omega},}
\end{array}\right.
\]
satisfies
\[
\min_{\overline{\Omega'}}\left|\nabla u\right|>0.
\]

\end{lem}
By using a standard compactness argument it is possible to give a
quantitative version of this result  \textcolor{black}{(see also \cite{2015-alessandrini-nesi})}. We restrict ourselves to  a
particular choice for $\phi$.
\begin{prop}
\label{prop:aless-magn}Let $\Omega\subseteq\R^{2}$ be a smooth,
bounded and convex domain and take $\Omega'\Subset\Omega$. Let $a\in\Cl^{0,\alpha}(\overline{\Omega};\R^{2\times2})$
be such that \eqref{eq:ellipticity_a-multi} and $\left\Vert a\right\Vert _{\Cl^{0,\alpha}(\overline{\Omega};\R^{2\times2})}\le C_{1}$
hold true for some $C_{1}>0$. Take $\beta\in\R$ with $\left|\beta\right|\le C_{1}$.
The solution $u\in\Cl^{1}(\overline{\Omega})$ to 
\[
\left\{ \begin{array}{l}
-\div(a\nabla u)=0\qquad\text{in \ensuremath{\Omega},}\\
u=x_{1}+\beta x_{2}\qquad\text{on \ensuremath{\partial\Omega},}
\end{array}\right.
\]
satisfies
\[
\min_{\overline{\Omega'}}\left|\nabla u\right|\ge C
\]
for some $C>0$ depending only on $\Omega$, $\Omega'$, $\maxa$,
$\alpha$ and $C_{1}$.\end{prop}
\begin{rem}
Under the assumption $a\in\Cl^{0,1}$, it is possible to give an explicit
expression for the constant $C$ \cite[Remark 3]{alessandrini-1988}.\end{rem}
\begin{proof}
By contradiction, assume that there exist two sequences $a_{n}\in\Cl^{0,\alpha}(\overline{\Omega};\R^{2\times2})$
and $\beta_{n}\in\R$ such that $a_{n}$ satisfies \eqref{eq:ellipticity_a-multi},
$\left\Vert a_{n}\right\Vert _{\Cl^{0,\alpha}(\overline{\Omega};\R^{2\times2})}\le C_{1}$,
$\left|\beta_{n}\right|\le C_{1}$ and
\[
\min_{\overline{\Omega'}}\left|\nabla u_{n}\right|\to0,
\]
where $u_{n}$ is the unique solution to
\[
\left\{ \begin{array}{l}
-\div(a_{n}\nabla u_{n})=0\qquad\text{in \ensuremath{\Omega},}\\
u=x_{1}+\beta_{n}x_{2}\qquad\text{on \ensuremath{\partial\Omega}.}
\end{array}\right.
\]
Take $x_{n}\in\overline{\Omega'}$ such that $\left|\nabla u_{n}(x_{n})\right|\to0$.
Up to a subsequence, we have that $x_{n}\to\tilde{x}$ for some $\tilde{x}\in\overline{\Omega'}$
and $\beta_{n}\to\tilde{\beta}$ for some $\tilde{\beta}\in[-C_{1},C_{1}]$.
By the Ascoli-Arzelà theorem, the embedding $C^{0,\alpha}\hookrightarrow\Cl^{0,\alpha/2}$
is compact. Thus, up to a subsequence, we have that $a_{n}\to\tilde{a}$
in $\Cl^{0,\alpha/2}(\overline{\Omega};\R^{2\times2})$ for some $\tilde{a}\in\Cl^{0,\alpha/2}(\overline{\Omega};\R^{2\times2})$
satisfying \eqref{eq:ellipticity_a-multi} and $\left\Vert \tilde{a}\right\Vert _{\Cl^{0,\alpha/2}(\overline{\Omega};\R^{2\times2})}\le C(\Omega)C_{1}$.

Let $\tilde{u}$ be the unique solution to
\[
\left\{ \begin{array}{l}
-\div(\tilde{a}\nabla\tilde{u})=0\qquad\text{in \ensuremath{\Omega},}\\
\tilde{u}=x_{1}+\tilde{\beta}x_{2}\qquad\text{on \ensuremath{\partial\Omega}.}
\end{array}\right.
\]
By looking at the equation satisfied by $u_{n}-\tilde{u}$, by Proposition~\ref{pro: Helmholtz regularity}
it is easy to see that $\left\Vert u_{n}-\tilde{u}\right\Vert _{\Cl^{1}(\overline{\Omega};\R)}\to0$.
Therefore
\[
\left|\nabla\tilde{u}(\tilde{x})\right|\le\left|\nabla\tilde{u}(\tilde{x})-\nabla\tilde{u}(x_{n})\right|+\left|\nabla\tilde{u}(x_{n})-\nabla u_{n}(x_{n})\right|+\left|\nabla u_{n}(x_{n})\right|\to0,
\]
whence $\left|\nabla\tilde{u}(\tilde{x})\right|=0$, which contradicts
Lemma~\ref{lem:ales-magn}, as $\Omega$ is  convex.\end{proof}

\end{document}